\documentclass[11pt]{article}
\usepackage{graphics,graphicx,subfigure}
\usepackage{color}
\usepackage{hyperref}
\usepackage{amssymb}
\usepackage{amsthm,amsmath}
\usepackage[authoryear,round]{natbib} %

\numberwithin{equation}{section}

\newtheorem{theorem}{Theorem}[section]
\newtheorem{proposition}{Proposition}[section]
\newtheorem{corollary}{Corollary}[section]

\theoremstyle{remark}
\newtheorem{remark}{Remark}[section]

\begin{document}

\title{Response of an oscillatory differential delay equation to a single stimulus \thanks{This
        work was supported by the Natural Sciences and Engineering Research Council (NSERC) of Canada and the Polish NCN grant no 2014/13/B/ST1/00224.}}
\author{Michael C. Mackey
\thanks{Departments of Physiology, Physics \& Mathematics, McGill University, 3655 Promenade Sir William Osler, Montreal, Quebec H3G 1Y6, Canada ({\tt michael.mackey@mcgill.ca}).}%}
        \and Marta Tyran-Kami\'nska
        \thanks{Institute of Mathematics, University of Silesia, Bankowa 14, 40-007 Katowice, Poland
        ({\tt mtyran@us.edu.pl}.)}
        \and Hans-Otto Walther
        \thanks{Mathematisches Institut, Universit{\"{a}}t Giessen, Arndtstrasse 2, 35392 Giessen, Germany
        ({\tt Hans-Otto.Walther@math.uni-giessen.de}.)}
        }

\date{\today}

%%for the svjour3
%\author{Michael C. Mackey \and Marta Tyran-Kami\'nska \and Hans-Otto Walther
%\institute{Michael C. Mackey \at Departments of Physiology, Physics \& Mathematics and Centre for Applied Mathematics in Bioscience and Medicine (CAMBAM), McGill University, 3655 Promenade Sir William Osler, Montreal, QC, CANADA, H3G 1Y6
%\\ \email{michael.mackey@mcgill.ca}\and Marta Tyran-Kami\'nska \at Institute of Mathematics,
%University of Silesia, Bankowa 14, 40-007 Katowice, POLAND\\ \email{mtyran@us.edu.pl} \and Hans-Otto Walther
%        \at Mathematisches Institut, Universit{\"{a}}t Giessen, Arndtstrasse 2, 35392 Giessen, GERMANY\\
%        \email{Hans-Otto.Walther@math.uni-giessen.de}}

\maketitle

\begin{abstract}Here we analytically examine the response of a limit cycle solution to a simple differential delay equation to a single pulse perturbation of the piecewise linear nonlinearity.  We construct the unperturbed limit cycle analytically, and are able to completely characterize the perturbed response to a pulse of positive amplitude and duration with onset at different points in the limit cycle.  We determine the perturbed minima and maxima and period of the limit cycle and show how the pulse modifies these from the unperturbed case.
   \end{abstract}
%\begin{keywords}
%delay differential equation, perturbation, phase resetting, phase locking, synchronization
%\end{keywords}

%\begin{AMS}
%34K05, 34K11, 34K13, 34K18, 34K26, 34K27
%\end{AMS}
%

\tableofcontents

\section{Introduction}\label{sec:intro}

Mammalian hematopoietic systems have complex and complicated regulatory processes that control the production of red blood cells, white blood cells and platelets.  However, boiled down to their essence, each is a negative feedback system with a time delay that is controlling the production of primitive cells entering from the hematopoietic stem cell compartment.

Often the numbers of circulating blood cells will display oscillations that are more or less regular.  This may occur \citep{foley-mackey-2009} because of the existence of a spontaneously occurring disorder like cyclical neutropenia \citep{dale88,haurie98a,Colijn2007,dalemackey2015}, periodic thrombocytopenia \citep{apostu2008,Swinburne2000}, periodic leukemia \citep{pcml,fortin}, or periodic autoimmune hemolytic anemia \citep{Mac3,jgmmcm89}.  Or, it may occur because of the intrusive administration of chemotherapy in a periodic fashion \citep{krinner2013} which has the unfortunate side effect of killing both malignant and normal cells.

In either case (spontaneously occurring oscillations due to disease or induced oscillations due to the side effects of chemotherapy) a clinical intervention  often consists of trying to administer a recombinant cytokine of the appropriate type to alleviate the more serious symptoms of the oscillation.  In the case of cyclical neutropenia this is granulocyte colony stimulating factor (G-CSF) \citep{dale88,dale93,Dale2003}, and the same is true during chemotherapy induced neutropenia \citep{bennett1999,clark2005}. G-CSF has among its effects the ability to interfere with apoptosis (pre-programmed death) of cells \citep{Colijn2007,colijn-foley-mcm-2007}, be this cell death naturally present or induced by an external agent like chemotherapy. Unfortunately  the issue of when to administer this (or other) cytokines is hotly debated and this is, without a doubt, because the cytokines in question have an effect on the dynamics of the affected system many days before the desired (or undesired) effect is manifested in the peripheral blood.

It has been noted that the timing of the administration  of G-CSF can have profound consequences on the neutropenia.  Given at some points in the cycle it can dramatically reduce the neutropenia (increasing the nadir of the cycle) while at other times it can actually make the neutropenia worse by deepening the nadir \citep{aapro2011,barni2014,langa2012,palumbo2012}.  Thus, from a mathematical perspective the problem is simply ``How and when do we deliver a perturbation to a delayed dynamical system in order to achieve some desired objective?".  %This paper is directed to a partial investigation of this problem.

The problem outlined above can, from a mathematical point of view, be viewed within the context of `phase resetting of an oscillator' and as such has received widespread attention especially within the biological community.  This field has a large and varied literature (see \citet{clocks} for an elementary introduction and \citet{winfree80} for an exhaustive treatment of the subject from a historical perspective) which is almost exclusively devoted to the interaction of oscillatory systems in a finite dimensional space (i.e. limit cycles in ordinary differential equations) with a single perturbation or periodic  perturbation.  Surprisingly, however, there is little that has been done on such interactions when the limit cycle is in an infinite dimensional phase space (e.g a differential delay equation).  There are, however, a few authors who have considered such situations. % but only for a periodically applied perturbation and not the case of a single perturbation studied here.

For example, \citet{forys2013,bodnar2013a,bodnar2014} and \citet{forys2014} studied  simple models of tumour growth where the delayed model equation has an additional term describing an external influence and reflecting a treatment.  There have been a number of both experimental and theoretical papers \citep{israelsson1967,johnsson1968,johnsson1971,andersen1972a,andersen1972b} devoted to the autonomous growth of the tip of {\it Helianthus annuus} which describes a variety of patterns as a function of time and which is thought to involve a delay between the sensing of a gravitational stimulus and the bending of the plant (c.f \citet{israelsson1967} for a very nice historical review of this problem).
   Another class of problems involving delayed dynamics is related to pulse coupled oscillators which have been treated recently by \citet{canavier2010,klinshov2011,klinshov2015}.
\citet{bard2012} and \citet{novicenko2012} have developed phase reduction methods appropriate for delayed dynamics.
 Finally we should note the recent numerical work on several gene regulatory circuite by \citet{lewis2003} and \citet{horikawa2006} for the segmentation clock in zebrafish as well as the work of \citet{doi2010} on circadian regulation of G-protein signaling.  However, none of these papers have addressed the problem that we study here from an analytic point of view. This paper offers a partial study of the problem.

The regulation of the production of blood cells, denoted by $x(t)$ (and typically measured in units of cells/$\mu$L of blood or alternately in units of cells/kG body weight), reduced to the barest of descriptions, can be described most simply by a differential delay equation of the form
\begin{equation}
\dfrac{dx}{dt}(t) = -\gamma x(t) + f(x(t-\tau)) \qquad \mbox{with constant} \,\, \gamma > 0,
\label{e:gen-dde}
\end{equation}
in which $f:\mathbb{R}\to\mathbb{R}$ is monotone decreasing such that $\xi_1 \leq \xi_2$ implies that $f(\xi_1) \geq f(\xi_2)$. In Equation \eqref{e:gen-dde} we must also specify an initial function $\varphi: [-\tau,0]\to\mathbb{R}$, in order to obtain a solution.  Here we replace the nonlinearity $f$ with a piecewise constant function. This permits us to compute solutions explicitly, so we may analytically study their behaviour, and the response of the solutions to perturbation meant to represent the effect of cytokine administration.

This generic model captures the essence, if not the subtleties, of peripheral blood production.  The monotone nature of $f$ is mediated via the effects of the important regulatory cytokines, e.g.  G-CSF for the white blood cells, erythropoietin for the red blood cells, and thrombopoietin for the platelets.  The administration of exogenous cytokines in an attempt to control the dynamics of \eqref{e:gen-dde} will typically have an effect that be interpreted as increasing $f$ over some portion of time, and the goal of this paper is to study the effect of such a perturbation on the solution of  \eqref{e:gen-dde}.

This paper is organized as follows.  Section \ref{sec:form} describes the model and formulates it in a mathematically convenient form. Section \ref{sec:global} provides basic facts about continuous, piecewise differentiable solutions. On a state space of simple initial functions these solutions yield a continuous semiflow. There is a periodic solution whose orbit in state space is stable with strong attraction properties. Section \ref{sec:pulse} introduces the  pulse-like perturbations (a perturbation of constant amplitude $a$ lasting for a finite period of time $\sigma$) of the model which correspond to the effect of medication in the sense that during a finite time interval the production of blood cells is increased. It is shown that the response of the system to such perturbations is continuous provided the latter are not too large. This includes a continuity result for the {\it cycle length map}, which assigns to each onset time of increased production a time of return to the periodic orbit, after the end of increased production.
The bulk of our results are presented in Section \ref{sec:pert} where we examine the effect of cytokine perturbation
when the perturbation away from the stable periodic orbit begins at different points in the cycle.  In particular we look at phase resetting properties of the system - in terms of the cycle length map - and at the minima and maxima compared to the amplitudes of the periodic solution.  Section \ref{sec:prop-reset} examines the various forms assumed by the cycle length map for different values of the parameters. Section \ref{sec:therapy} shows how the results of the previous sections may be potentially used to tailor  therapy to achieve certain results.  The paper concludes with a brief discussion in Section \ref{sec:disc}.  There we consider a simple extension in which a pulse-like perturbation may decrease the nadir of the limit cycle as is noted clinically.   The proofs of many of the results are given in the two appendices.

\section{The model}\label{sec:form}

\subsection{Scalar delay differential equations}

Consider the  delay differential equation \eqref{e:gen-dde}.
If $f$ is continuous and monotone decreasing then
there is a unique constant solution $t\mapsto x_*$ given by $f(x_*) = \gamma x_*$. If in addition $f$ is, say, continuously differentiable then this constant solution may be stable or not, depending on $\gamma$ and $f'(x_*)$.
In case $t\mapsto x_*$ is {\it linearly unstable}, that is, the {\it linearized equation}
$$
y'(t)=-\gamma y(t)+f'(x_*)y(t-\tau)
$$
is unstable then also  $t\mapsto x_*$ is unstable as a solution of \eqref{e:gen-dde}. In case $f$ also satisfies a one-sided boundedness condition there exists a periodic solution which is slowly oscillating in the sense that the intersections with the equilibrium level $\xi=x_*$ are spaced at distances larger than the delay $\tau$, and the minimal period is given by three consecutive such intersections. In general slowly oscillating periodic solutions are not unique in the sense that they are not all translates of each other. Depending on $\gamma$ and $f'(x_*)$ there may also exist {\it rapidly oscillating} periodic solutions about $\xi=x_*$. These have all consecutive zeros spaced at distances strictly less than the delay $\tau$. In case $f'(\xi)<0$ for all $\xi\in\mathbb{R}$ every rapidly oscillating periodic solution is unstable. For details and for more about \eqref{e:gen-dde}, see e.g. the recent survey by \cite{W}   and the references given there.

\subsection{A piecewise constant approximation of the nonlinearity}\label{ssec:p-linear}

To obtain solutions which can be computed in terms of elementary functions consider the situation in which the function $f$ is piecewise constant and given by
\begin{equation}
f(x) =
\left\{\begin{array}{ll}  b_L  & \qquad \mbox{for} \quad x  <
    \theta  \\
    b_U   & \qquad \mbox{for} \quad x  \geq \theta,
    \end{array}
    \right.
    \quad 0<\theta,\quad 0 < b_U <   b_L.
    \label{e:pcnl}
\end{equation}
%We exclude the special case $b_U=\gamma\theta$, in order to facilitate the construction of the solution semi-flow in Section 3 below, see the dashed line in Figure \ref{fig:figf-1}.
We exclude the special cases $b_L=\gamma \theta$ and $b_U=\gamma\theta$, in order to facilitate the construction of the solution semi-flow in Section 3 below, see the skewed dashed lines in Figure \ref{fig:figf-1} and  Remark~\ref{rem:semi}.

\begin{figure}[htb]
\centering

\includegraphics{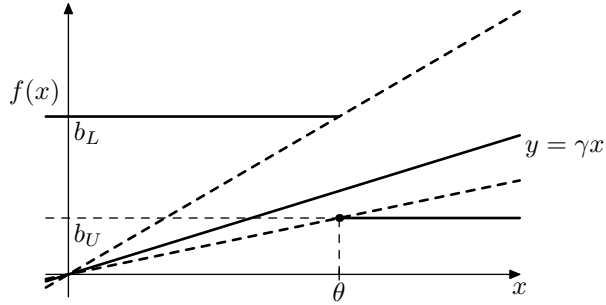}
\caption{The graph of the function $f$ as given in \eqref{e:pcnl}.  The skewed solid line is the graph of $\gamma x$ in general, while the skewed dashed lines are the graphs of $\gamma x$ in the special excluded cases that $b_L=\gamma \theta$ and  $b_U = \gamma \theta$.}
\label{fig:figf-1}
\end{figure}

Then we obtain the delay differential  equation
\begin{equation}
    \dfrac{dx}{dt}(t) = - \gamma x(t) + \left\{\begin{array}{ll}  b_L  & \qquad \mbox{for} \quad x(t-\tau) <
    \theta  \\
    b_U   & \qquad \mbox{for} \quad x(t-\tau) \geq \theta,
    \end{array}
    \right.
      \label{e:pl-sys}
    \end{equation}
with $\gamma>0,\theta>0,b_L>b_U>0,\gamma\theta\neq b_U,b_L$
whose solutions satisfy linear, inhomogeneous ordinary differential equations in  intervals on which their delayed values remain either below or above the level $\xi=\theta$, resulting in increasing and decreasing exponentials on such intervals. In the situation illustrated in Figure~\ref{fig:figf-1} when the graph of $x\mapsto \gamma x$  passes through the gap of the nonlinearity in Eq.~\eqref{e:pl-sys}, so that $b_U<\gamma\theta<b_L$, there is no steady state of Eq.~\eqref{e:pl-sys}. Otherwise, if either $\gamma\theta <b_U$ or $b_L<\gamma\theta$ then $x(t)=b_U/\gamma$ or $x(t)=b_L/\gamma$ is the steady state of Eq.~\eqref{e:pl-sys}.

Note that in \eqref{e:pl-sys} there are five parameters $(\gamma, b_U, b_L, \theta, \tau)$ and we can reduce these to three by a change of variables, since with
\begin{equation}
\label{eq:var}
\left\{
\begin{array}{rcl}
 \hat{x}(t)& =& x(t/\gamma)-\theta \\
 \hat{\tau}&= &\gamma \tau \\
 \beta_L &=& -\theta +b_L/\gamma \\
 \beta_U &=& \theta -b_U /\gamma
\end{array}\right.
\end{equation}
we can rewrite \eqref{e:pl-sys} in the form
\begin{equation}
    \dfrac{d\hat{x}}{dt}(t) =  \left\{\begin{array}{ll}-  \hat{x}(t)  + \beta_L  & \quad \mbox{for} \quad \hat{x}(t-\hat{\tau}) <
    0 \\
    -  \hat{x}(t) -\beta _U   & \quad \mbox{for} \quad \hat{x}(t-\hat{\tau}) \geq 0,
    \end{array}
    \right.
       \label{e:pl-sys-dimen}
    \end{equation}
where $\beta_L+\beta_U>0$.

Now we have only a three parameter $(\beta_U, \beta_L, \tau)$ system to consider, having  reduced \eqref{e:pl-sys} to
 \begin{equation}
    \dfrac{dx}{dt}(t) =  -  x(t)  + f(x(t-\tau))
    \label{e:pl-sys-dimen1}
    \end{equation}
where the function $f$ is of the form
\begin{equation}
f(\xi) =
\left\{\begin{array}{ll}  \beta_L  & \qquad \mbox{for} \quad \xi  <
    0  \\
    -\beta_U   & \qquad \mbox{for} \quad \xi  \geq 0,
    \end{array}
    \right.
    \quad -\beta_U<\beta_L, \quad \beta_L,\beta_U\neq 0. %\quad 0<\beta_L.
    \label{e:pcnl1}
\end{equation}
The discontinuity of $f$ requires a moment of reflection about what a solution of Eq. \eqref{e:pl-sys-dimen1} should be - certainly not a continuous function $x:[-\tau,\infty)\to\mathbb{R}$ which is differentiable and satisfies Eq.~\eqref{e:pl-sys-dimen1} for all $t>0$, as it is familiar from delay differential equations with a functional on the right hand side which is at least continuous. We shall come back to this in Section 3.

\medskip

\section{The semiflow of the unperturbed system}\label{sec:global}

For continuity properties, e.g., continuity of the reset map which is to be introduced in Section~\ref{sec:pulse}, we need to develop a framework.

Consider the initial value problem of Eq.~\eqref{e:pl-sys-dimen1} for $t>0$, with initial condition $x(t)=\phi(t)$ for $-\tau\le t\le0$, and where the function $\phi:[-\tau,0]\to\mathbb{R}$ is continuous and has at most a finite number of zeros.
Let $C=C([-\tau,0],\mathbb{R})$ denote the Banach space of continuous functions $[-\tau,0]\to\mathbb{R}$ equipped with the maximum-norm, $|\phi|_C=\max_{-\tau\le t\le0}|\phi(t)|$, and set
$$
Z=\{\phi\in C:\phi^{-1}(0)\,\,\text{finite or empty}\}.
$$
For initial data $\phi\in Z$ we construct continuous solutions of Eq. \eqref{e:pl-sys-dimen1} by means of the variation-of-constants formula as follows. Suppose $z_1<z_2<\ldots<z_J$ are the zeros of $\phi$ in $(-\tau,0)$. On $(0,z_1+\tau]$ we define
$$
x(t)=e^{-t}\phi(0)+\int_0^te^{-(t-s)}f(x(t-s))ds
$$
or,
\begin{equation}
x(t)=e^{-t}x(0)-\beta_U(1-e^{-t})=-\beta_U+(x(0)+\beta_U)e^{-t}
\label{e:3.8}
\end{equation}
in case $0<\phi(v)$  in $(-\tau,z_1)$,
\begin{equation}
x(t)= e^{-t}x(0)+\beta_L(1-e^{-t})=\beta_L+(x(0)-\beta_L)e^{-t}
\label{e:3.9}
\end{equation}
in case $\phi(v)<0$ in $(-\tau,z_1)$.
Notice that on the interval $[0,z_1+\tau]$ the solution $x$ is either constant with value $-\beta_U\neq0$ or $\beta_L\neq 0$, or strictly monotone. We conclude that $x$ has at most one zero in $(0,z_1+\tau)$, and if there is a zero then $x$ changes sign at the zero. Let us call such zeros {\it transversal}. In case $\phi$ has no zero in $(-\tau,0)$ we define $x(t)$ analogously, for $0<t\le\tau$.
The procedure just described can be iterated and yields a continuous function $x\colon [-\tau,
\infty)\to\mathbb{R}$ which we define to be the solution of the initial value problem above.

Notice that all {\it segments}, or {\it histories}, $x_t$ given by
$$
x_t(s)=x(t+s)\,\,\text{for}\,\,t\ge0\,\,\text{and}\,\,-\tau\le s\le0,
$$
belong to the set $Z$. We assume that $\beta_L,\beta_U\neq0$  and we write $x^{\phi}$ instead of $x$ when convenient, and define the semiflow $S\colon [0,\infty)\times Z\to Z$ of Eq. \eqref{e:pl-sys-dimen1} by
$$
S(t,\phi)=x^{\phi}_t.
$$
The proof of the following result is given in the appendix.
\begin{proposition}\label{prop:prop3.1}
The semiflow $S$ is continuous.
\end{proposition}

\begin{remark}\label{rem:semi}
Incidentally, let us see what goes wrong in the excluded cases $\beta_U=0$ and $\beta_L=0$. If $\beta_U=0$ then for each $\phi\in Z$ with $\phi(t)>0$ on $[-\tau,0)$ and $\phi(0)=0$ the formula (\ref{e:3.8}) - or, the equation $x'=-x$ - yields $x(t)=0$ for all $t\ge0$, and $Z$ is not positively invariant. On a set $\tilde{Z}\subset C$ of initial data which contains $0\in C$ and negative data $\psi$ with arbitrarily small norm continuous dependence on initial data would be violated for $\beta_U=0$ because
(\ref{e:3.9}) yields
\begin{align*}
|x(\tau)-0| & =  |x(\tau)|=|\beta_L+(x(0)-\beta_L)e^{-\tau}|\\
& \ge  \beta_L(1-e^{-\tau})-\psi(0)e^{-\tau}\ge
\dfrac{1}{2}\beta_L(1-e^{-\tau})
\end{align*}
for the solution starting from negative and sufficiently small $\psi\in\tilde{Z}$.

If $\beta_L=0$ then for each $\phi\in Z$ with $\phi(t)<0$ on $[-\tau,0)$ and $\phi(0)=0$ we have $x(t)=0$ for $t\in [0,\tau]$, which implies that $Z$ is not positively invariant in this case as well.
\end{remark}

For later use we show next that transversal zeros depend continuously on the initial data $\phi\in Z$.

\begin{proposition}\label{prop:3.2}
For $\phi\in Z$ and $z>0$ with $x^{\phi}(z)=0\neq x^{\phi}(z-\tau)$ and $\epsilon>0$, there exists $\delta>0$ such that for each $\psi\in Z$ with $|\psi-\phi|_C<\delta$ there is $z'\in(z-\epsilon,z+\epsilon)$ with
$x^{\psi}(z')=0$. Moreover, $x^{\psi}(z'-\tau)\neq 0$.
\end{proposition}

\begin{proof} By continuity there exists $\eta\in(0,\epsilon)$ so that $x^{\phi}(s)\neq0$ on $[z-\tau-\eta,z-\tau+\eta]\cap[-\tau,\infty)$. Using (\ref{e:3.8}) and (\ref{e:3.9}) we infer that on $[z-\eta,z+\eta]$ the solution $x^{\phi}$ either equals a nonzero constant or is strictly monotone. As $x^{\phi}(z)=0$ the solution $x^{\phi}$ must be strictly monotone on $[z-\eta,z+\eta]$, with $\mathrm{sign}(x^\phi(z-\eta))\neq \mathrm{sign}(x^{\phi}(z+\eta))\neq0$. By continuous dependence on initial data, there exists $\delta>0$ such that for each $\psi\in Z$ with $|\psi-\phi|_C<\delta$ we have
$$
x^{\psi}(t)\neq0\,\,\text{on}\,\,[z-\tau-\eta,z-\tau+\eta]\cap[-\tau,\infty)
$$
and
$$
\mathrm{sign} (x^{\psi}(z-\eta))=\mathrm{sign} (x^{\phi}(z-\eta)),\,\,\mathrm{sign} (x^{\psi}(z+\eta))=\mathrm{sign} (x^{\phi}(z+\eta)).
$$
Hence $x^{\psi}$ changes sign in $[z-\eta,z+\eta]$, so it has a zero $z'$ in this interval.
\end{proof}

The condition $\beta_L<0$ ($\beta_U<0$) in the next result means that in the original model given by Eq. \eqref{e:pl-sys}
the positive constant solution given by $\gamma x^{\ast}=f(x^{\ast})$ has its value  $x^{\ast}$ beyond (above) the discontinuity $\theta$ of $f$. In this case one may interpret $x^{\ast}=\theta$ as an equilibrium position for (\ref{e:pl-sys}).

\begin{theorem}
If $\beta_U<0$ or $\beta_L<0$ the equilibrium state of the semi-flow $S$, which is respectively given by $x^\phi(t)=-\beta_U$ or $x^\phi(t)=\beta_L$ with $\phi \in Z$,
is globally asymptotically stable.
\end{theorem}

\begin{proof}
In case $\beta_U < 0$ the constant function $\mathbb{R}\ni t\mapsto -\beta_U\in\mathbb{R}$ is a positive solution.
Notice that due to (\ref{e:3.8}) every solution $x^{\phi}$ with $0<\phi(t)$ on $[-\tau,0]$ has its values $x^{\phi}(t)$, $t\ge0$, between $\phi(0)$ and $-\beta_U$ and converges to $-\beta_U$ as $t\to\infty$. This implies local asymptotic stability of the positive steady state. Moreover, using (\ref{e:3.8}) and (\ref{e:3.9}) one can show that every solution $x=x^{\phi}$, $\phi\in Z$, becomes positive on some interval $[\tilde{T}-\tau,\tilde{T}]$, $\tilde{T}\ge0$, and is  given by
$$
x(t)=-\beta_U+(x(\tilde{T})+\beta_U)e^{-(t-\tilde{T})}
$$
for $t\ge \tilde{T}$, so it tends to $-\beta_U$ as $t\to\infty$. The proof in case $\beta_L<0$ is analogous.  \end{proof}

The situation becomes more interesting when  $\beta_L,\beta_U >0$ which gives $b_U<\gamma\theta<b_L$ for the original parameters, so that the model equation (\ref{e:pl-sys}) does not have a steady state, being a constant function.  
We will show in the next two theorems that in this case there exists a periodic solution and that it is stable, when the semiflow is restricted to
the smaller set $Z_0\subset Z$  of all $\phi\in Z$ which have at most one zero and change sign at this zero $z$  in case $-\tau<z<0$.

We first show that $Z_0$ is positively invariant for the semiflow $S$, i.e., $S(t,\phi)\in Z_0$ for all $t\ge 0$ and $\phi\in Z_0$.  Suppose that $\beta_L>0$ and $\beta_U>0$.  For $\phi\in Z_0$ and $x=x^{\phi}$
we make the following observations. If $\phi(z)=0$ and $-\tau<z<0$, and $\phi(t)<0$ for $-\tau\le t<z$ and $0<\phi(t)$ for $z<t\le0$ then by (\ref{e:3.9}), $0<x(t)$ on $(z,z+\tau]=(z,0]\cup(0,z+\tau]$. In particular, $x_{z+\tau}\in Z_0$; moreover, $x_t\in Z_0$ for all $t\in[0,z+\tau]$. Using (\ref{e:3.8}) there is a smallest $z'$ in $(0,\infty)$ with $x(z')=0$, $x(t)\neq 0$ on $[z'-\tau,z')$, and $x$ changes sign at $t=z'$, and we can iterate.
Thus we obtain $S(t,\phi)\in Z_0$ for all $t\ge0$. The same holds for arbitrary $\phi\in Z_0$.

Note that the zeros of $x^{\phi}$, $\phi\in Z_0$ arbitrary, in $(0,\infty)$ are all transversal and form a strictly increasing sequence of times $z_j=z_j(\phi)$, $j\in\mathbb{N}$, with
$$
z_j+\tau<z_{j+1}\,\,\text{and}\,\,x(z_j-\tau)\neq0\,\,\text{for all}\,\,j\in\mathbb{N}.
$$
From Proposition~\ref{prop:3.2} we conclude the following concerning solutions $x\colon[-\tau,\infty)\to\mathbb{R}$ starting from initial data $x_0=\phi\in Z_0$.

\begin{corollary}
Let $\beta_L,\beta_U>0$. Then each map
$$
Z_0\ni\phi\mapsto z_j(\phi)\in(0,\infty),\quad j\in\mathbb{N},
$$
is continuous at every point $\phi\in Z_0$ with $\phi(0)\neq0$. 
\end{corollary}

We now show the existence of a periodic solution.

\begin{theorem}
Let $\beta_L,\beta_U >0$. Then there is a periodic solution $\tilde{x}:\mathbb{R}\to\mathbb{R}$ of Eq. \eqref{e:pl-sys-dimen1} with
$$
\tilde{x}(t)=-\beta_U+\beta_Ue^{-(t+\tau)}\,\,\text{for}\,\,-\tau\le t\le0.
$$
We have $\tilde{x}(-\tau)=0$, and with %$\tilde{z}_j=z_j(\tilde{x}_0)=z_{\ast,j}(\tilde{x}_0)$
$\tilde{z}_j=z_j(\tilde{x}_0)$ for all $j\in\mathbb{N}$,
\begin{align*}
\tilde{x}(t) & =  \beta_L+(\tilde{x}(0)-\beta_L)e^{-t}\qquad\text{for}\quad 0\le t\le\tilde{z}_1+\tau,\\
\tilde{x}(t) & =  -\beta_U+(\tilde{x}(\tilde{z}_1+\tau)+\beta_U)e^{-(t-(\tilde{z}_1+\tau))}\qquad\text{for}\quad\tilde{z}_1+\tau\le t\le\tilde{z}_2+\tau.
\end{align*}
The minimal period of $\tilde{x}$ is
$$
\tilde{T}=\tilde{z}_2+\tau.
$$
\end{theorem}

\begin{figure}[htb]
\centering

\includegraphics{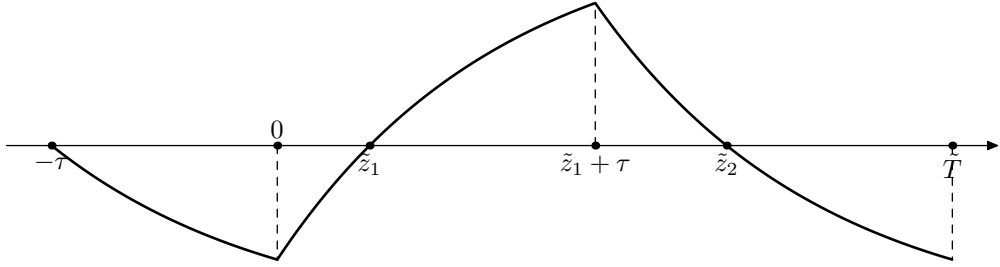}
\caption{A graph of the periodic solution on the interval $[-\tau,\tilde{T}]$}
\label{fig:figf-3}
\end{figure}

\begin{proof}
Compute the solution (see Figure \ref{fig:figf-3}) starting from $\phi\in Z_0$ given by
\[
\phi(t)=-\beta_U+\beta_Ue^{-(t+\tau)}.\qedhere
\]
\end{proof}

It is convenient to set $\tilde{z}_0=-\tau$. %\qquad(=z_{\ast,0}(\tilde{x}_0))$.

\begin{corollary} Let $\beta_L,\beta_U>0$.
The minimum of $\tilde{x}$ is given by
$$
\underline{x}=\tilde{x}(0)= -\beta_U(1- e^{-\tau})<0,
$$
the maximum of $\tilde{x}$ is given by
$$
\overline{x}=\tilde{x}(\tilde{z}_1+\tau)=\beta_L (1- e^{- \tau})>0,
$$
and
$$
\tilde{z}_1=\ln \dfrac{\beta_L -  \underline x }{\beta_L  },
\qquad \tilde{z}_2=\tilde{z}_1+\tau+\ln\dfrac {\overline x + \beta_U}{ \beta_U},
$$
$$
\tilde{T}=\tilde{z}_2+\tau=2\tau+\ln
\left(\dfrac{\beta_L -  \underline x }{\beta_L  } \cdot  \dfrac {\overline x + \beta_U}{ \beta_U}\right) .
$$
\label{cor:max-min-per}
\end{corollary}

Occasionally we shall abbreviate
$$
t_{\max}=\tilde{z}_1+\tau
$$
for the first nonnegative time where $\tilde{x}$ achieves its maximum. Accordingly,
$$
t_{\min}=0.
$$

\begin{remark}
For  $\beta_U >0$ there are infinitely many other periodic orbits in $Z$, given by periodic solutions of higher oscillation frequencies, compare \citep[Chapter XVI]{DvGVLW}
which deals with the simpler equation
$$
x'(t)=-\mathrm{sign}(x(t-1)).
$$
Therefore the periodic orbit
$$
\tilde{O}=\{\tilde{x}_t\in Z:t\in\mathbb{R}\}\subset Z_0
$$
is not globally attracting on $Z$.
\end{remark}
\begin{theorem}
Let $\beta_L,\beta_U>0$. Then for every $\phi\in Z_0$ we have
either
$$
x^{\phi}(t+z_1(\phi)+\tau)=\tilde{x}(t)\qquad\text{for all}\quad t\ge-\tau
$$
or
$$
x^{\phi}(t+z_1(\phi)+\tau)=\tilde{x}(t+\tilde{z}_1+\tau)\qquad\text{for all}\quad t\ge-\tau,
$$
and the periodic orbit $\tilde{O}\subset Z_0$ is stable.
\label{t:3.3}
\end{theorem}

\begin{proof}
\begin{enumerate}
\item For $\phi\in Z_0$, $x=x^{\phi}$, and $z_1=z_1(\phi)$ we infer from (\ref{e:3.8}) and (\ref{e:3.9}) that
either $x$ is strictly decreasing on $[z_1,z_1+\tau]$, or that $x$ is strictly increasing on $[z_1,z_1+\tau]$. In the first case we obtain
$$
x(t+z_1+\tau)=\tilde{x}(t)\qquad\text{for all}\quad t\ge-\tau
$$
while in the second case,
$$
x(t+z_1+\tau)=\tilde{x}(t+\tilde{z}_1+\tau)\qquad\text{for all}\quad t\ge-\tau.
$$
In both cases, $x^{\phi}_t$ is on the periodic orbit $\tilde{O}$ for $t\ge z_1(\phi)+\tau$.

\item There exist $r>0$ and $\rho\in(0,\tilde{z}_1)$ with $\tilde{x}(t)\le-r$ for $\tilde{z}_1-\rho-\tau\le t\le \tilde{z}_1-\rho$ and
$\tilde{x}(\tilde{z}_1+\rho)\ge r$. It follows that for each $\psi\in\tilde{O}$ the shifted copy $x^{\psi}$ of $\tilde{x}$ has the property that there exists
$u=u(\psi)\in[0,\tilde{T}]$ with $x^{\psi}(t)\le-r$ for  $u\le t\le u+\tau$ and $x^{\psi}(u+\tau+2\rho)\ge r$. Observe that for every $\phi\in Z$ with $x^{\phi}(t)<0$ for $u\le t\le u+\tau$ and $x^{\phi}(u+\tau+2\rho)>0$ the solution $x^{\phi}$ has a first zero $z$ in $(u+\tau,u+\tau+2\rho)$, and $x^{\phi}_{z+\tau}\in\tilde{O}$, which implies that all segments $x^{\phi}_t$ with $t\ge\tilde{T}+2\tau+2\rho$ belong to the orbit $\tilde{O}$.

\item Let $\epsilon>0$. Continuous dependence on initial data yields that the map
$$
Z\ni\phi\mapsto x^{\phi}|_{[-\tau,\tilde{T}+2\tau+2\rho]}\in C([-\tau,\tilde{T}+2\tau+2\rho],\mathbb{R}),
$$
where $x^{\phi}|_{I}$ denotes the restriction of the function $x^{\phi}$ to the interval $I$,
is continuous (with respect to the maximum-norm on the target space). Using uniform continuity on the compact orbit $\tilde{O}\subset Z$ we find $\delta>0$ so that  for all $\phi\in Z$ and all $\psi\in\tilde{O}$ with $|\phi-\psi|_C<\delta$ we have
$$
|x^{\phi}(t)-x^{\psi}(t)|<\min\left\{\epsilon,\dfrac{r}{2}\right\}\qquad\text{for all}\quad t\in[-\tau,\tilde{T}+2\tau+2\rho].
$$
Let $\phi\in Z$ with $\text{dist}(\phi,\tilde{O})=\inf_{\psi\in\tilde{O}}|\phi-\psi|_C<\delta$ be given. Then for some $\psi\in\tilde{O}$, $|\phi-\psi|_C<\delta$.
Choose $u=u(\psi)\in[0,\tilde{T}]$ according to Part 2. The previous estimate of $|x^{\phi}(t)-x^{\psi}(t)|$
yields $x^{\phi}(t)<0$ for $u\le t\le u+\tau$ and $x^{\phi}(u+\tau+2\rho)>0$. Using part 2 we infer $x^{\phi}_t\in\tilde{O}$ for $t\ge\tilde{T}+2\tau+2\rho$.  Altogether,
$$
\text{dist}(x^{\phi}_t,\tilde{O})<\epsilon\qquad\text{for}\quad0\le t\le\tilde{T}+2\tau+2\rho
$$
and $\text{dist}(x^{\phi}_t,\tilde{O})=0$ for $t\ge\tilde{T}+2\tau+2\rho$.
\end{enumerate}
\end{proof}

\begin{remark}
One can show that the other periodic orbits are all unstable, and that the domain of attraction of the periodic orbit  $\tilde{O}$ is open and dense in $Z$, compare \citep[Chapter XVI]{DvGVLW}, and the main result of \cite{M-PW}
about equation (\ref{e:gen-dde}) with a smooth and strictly monotone function $f$.
\end{remark}

\section{Pulse-like perturbations}
\label{sec:pulse}

\medskip

In the following we assume
$$
-\beta_U<0<\beta_L
$$
and study a particular, simple deviation from the periodic solution  $\tilde{x}$ and the subsequent return to the stable and attracting periodic orbit $\tilde{O}$: We consider a function $x^{(\Delta)}\colon \mathbb{R}\to\mathbb{R}$ which up to $t=\Delta\in[0,\tilde{T})$ equals the periodic solution $\tilde{x}$ of Eq. \eqref{e:pl-sys-dimen1}. Then for $\Delta\le t\le\Delta+\sigma$  the function $x^{(\Delta)}$ is defined by the equation
$$
x'(t)=-x(t)+f(x(t-\tau))+a
$$
with a constant $a>0$. This results in a deviation from the periodic solution $\tilde{x}$ which begins at time $t=\Delta$ and lasts until the time $t=\Delta+\sigma$. Informally we speak of a pulse of amplitude $a$, with $\Delta$ the onset time of the pulse and  $\sigma$ the duration of the pulse. For $t\ge\Delta+\sigma$, the function $x^{(\Delta)}$ is given again by Eq. (\ref{e:pl-sys-dimen1}). For perturbations $a>0$ not too large it will merge into the periodic solution in finite time, compare Theorem \ref{t:3.3}.

\medskip

An interpretation of this is as follows. Eq. (\ref{e:pl-sys-dimen1}) is a mathematically convenient form of a (very simple) model for the production and decay of blood cells of a certain type, e.g. neutrophils.  The periodic solution $\tilde{x}$ stands for the density of neutrophils in a patient as a function of time, perhaps induced by chemotherapy or as a consequence of cyclical neutropenia but in the absence of any further medical intervention. The function $x^{(\Delta)}$ describes the evolution of the neutrophil density for the case that at time $\Delta-\tau$ some medication has been administered which increases the production of cells in the bone marrow during the time interval $[\Delta-\tau,\Delta-\tau+\sigma]$. The constant $a>0$ stands for the increase in production occasioned, for example, by the administration of G-CSF. After the time $\tau>0$ needed for production (and differentiation) of cells, that is, during the time interval $[\Delta, \Delta+\sigma]$ the neutrophils are released into the blood stream.  Later on production and decay of neutrophils is again governed by the patient's feedback system alone.

%\medskip

For simplicity we assume
$$
0<\sigma\le\tau,
$$
from here on, that is, the effect of intervention lasts for a time interval $\sigma$ less than the (production) delay $\tau$. The quantities we are interested in are the local extrema of $x^{(\Delta)}$ and the time required to return to the periodic orbit.  The latter is  captured by the {\it cycle length map}
$$
T:[0,\tilde{T})\ni\Delta\mapsto T(\Delta)\in\mathbb{R}\cup\{\infty\}
$$
which is defined formally as follows: The zeros of $x^{(\Delta)}$ and of $\tilde{x}$ in $(-\infty,\Delta]$ coincide. Suppose $\tilde{z}_J$, $J=j(\Delta)\in\{0,1,2\}$, is the largest one of these zeros, and there exists a smallest zero $z>\tilde{z}_J$ of $x^{(\Delta)}$ with $x^{(\Delta)}(z+t)=\tilde{x}(\tilde{z}_J+t)$ for all $t\ge0$. Then
$$
T(\Delta)=z-\tilde{z}_J,\quad\text{and}\quad T(\Delta)=\infty\quad\text{otherwise}.
$$
In case $T(\Delta)<\infty$ let
$$
\underline{x}_{\Delta}=\min_{\tilde{z}_J\le t\le z}x^{(\Delta)}(t)
$$
and
$$
\overline{x}_{\Delta}=\max_{\tilde{z}_J\le t\le z}x^{(\Delta)}(t).
$$

\begin{remark}
Observe that in our model situation treatment is considered successful if the minimal value $\underline{x}_{\Delta}$ is above the minimal value
$$
\underline{x}=\tilde{x}(0)=-\beta_U(1-e^{-\tau})<0\qquad\text{(see Corollary \ref{cor:max-min-per})}
$$
of $\tilde{x}$
while in the opposite case medication actually increases the risk for the patient in the sense that the nadir of the oscillation is lower and would thus lead to more severe cytopenia which is one of the major clinical problems.
\end{remark}
The  local minima and maxima of the function $x^{(\Delta)}$ and the cycle length $T(\Delta)$ depend on the parameters
$$
\tau>0,\,\,\beta_L>0>-\beta_U,\,\, a>0,\,\,\sigma\in(0,\tau],\quad\text{and}\quad \Delta\ge0.
$$
We assume that
$$
-\beta_U+a<0,
$$
 which will be instrumental in showing that after perturbation solutions do return to the periodic solution $\tilde{x}$, with the consequence that cycle lengths are finite. (For larger parameter $a$ solutions after perturbation may settle down on other, unstable periodic solutions of Eq. (\ref{e:pl-sys-dimen1}), with higher oscillation frequencies. For more on this, see Remark 5.4.)
In the remainder of this section we keep the parameters $\tau, \beta_L, \beta_U, a, \sigma$ fixed. In addition to finiteness of cycle lengths we shall see that the cycle length map and the maps
$$
[0,\tilde{T})\ni\Delta\mapsto\overline{x}_{\Delta}\in\mathbb{R}\quad\text{and}\quad
[0,\tilde{T})\ni\Delta\mapsto\underline{x}_{\Delta}\in\mathbb{R}
$$
are continuous.

For $a>0$ with $-\beta_U+a<0$ the solutions $x=x^{a,\phi}$ of the initial value problem
$$
x'(t)=-x(t)+f(x(t-\tau))+a\,\,\text{for}\,\,t>0,\,\,x_0=\phi
$$
define a continuous semiflow $S_a\colon[0,\infty)\times Z\to Z$ by $S_a(t,\phi)=x^{a,\phi}_t$, see Section~\ref{sec:global}.  The set $Z_0$ is positively invariant under $S_a$, and for each $\phi\in Z_0$ the zeros of $x^{a,\phi}$
are all transversal and spaced at distances larger than the delay $\tau$.

The solution $x=x^{(\Delta)} $ during a pulse (which begins at $\Delta\in[0,\tilde{T})$) can now be described as follows: For $\Delta$ given, define $\phi=\tilde{x}_{\Delta}$ and then $\chi=S_a(\sigma,\phi)=x^{a,\phi}_{\sigma}$. We obtain
\begin{align*}
x(t) & =  \tilde{x}(t)\,\,\text{for}\,\,t\le\Delta,\\
x(t) & =  x^{a,\phi}(t-\Delta)\,\,\text{for}\,\,\Delta\le t\le\Delta+\sigma,\\
x(t) & =  x^{\chi}(t-(\Delta+\sigma))\,\,\text{for}\,\,t\ge\Delta+\sigma,
\end{align*}
or equivalently,
\begin{align*}
x_t & =  \tilde{x}_t\,\,\text{for}\,\,t\le\Delta,\\
x_t & =  S_a(t-\Delta,x_{\Delta})\,\,\text{for}\,\,\Delta\le t\le\Delta+\sigma,\\
x_t & =  S(t-(\Delta+\sigma),x_{\Delta+\sigma})\,\,\text{for}\,\,t\ge\Delta+\sigma.
\end{align*}

Notice that for  $x=x^{(\Delta)}$, all zeros are transversal and spaced at distances larger than the delay $\tau$.
They form a strictly increasing sequence of points
$$
z_{\Delta,j},\quad j\in\mathbb{Z},\quad\text{with}\quad z_{\Delta,j}=\tilde{z}_j\quad\text{for all integers}\quad j\le j_{\Delta}
$$
where $J=j_{\Delta}\in\{0,1,2\}$ is given by
$$
\tilde{z}_J\le\Delta<\tilde{z}_{J+1}.
$$

Using $\tilde{x}_t=S(t,\tilde{x}_0)$ for all $t\ge0$ and the continuity of both semiflows we easily obtain from the
previous representation of  $x_t=x^{(\Delta)}_t$ that the map
$$
[0,\tilde{T})\times[0,\infty)\ni(\Delta,t)\mapsto x^{(\Delta)}_t\in C
$$
is continuous, which in turn yields the continuity of the map
$$
[0,\tilde{T})\times[0,\infty)\ni(\Delta,t)\mapsto x^{(\Delta)}(t)\in\mathbb{R}
$$
since $ x^{(\Delta)}(t)=ev(x^{(\Delta)}_t)$ and the evaluation $ev:C\ni\phi\mapsto\phi(0)\in\mathbb{R}$ is continuous.
Arguing as in the proof of Proposition \ref{prop:3.2} and using transversality of zeros we obtain the following results.

\begin{proposition}
For $\Delta_0\in[0,\tilde{T})$ and $z>0$ with $x^{(\Delta)}(z)=0$ and $\epsilon>0$ there exists $\delta>0$ such  that for each $\Delta\in[0,\tilde{T})$ with $|\Delta-\Delta_0|<\delta$ there is $z'\in(z-\epsilon,z+\epsilon)$ with
$x^{\Delta}(z')=0$.
\end{proposition}

\begin{corollary}\label{cor:contin}
Each map
$$
[0,\tilde{T})\ni\Delta\mapsto z_{\Delta,j}\in(0,\infty),\quad j\in\mathbb{N},
$$
is continuous.
\end{corollary}

The proofs of the following results are provided in the appendix.
\begin{proposition}\label{prop:prop4.2}
For every $\Delta\in[0,\tilde{T})$ we have $T(\Delta)=z_{\Delta,J+2}-z_{\Delta,J}$ with $J=j_{\Delta}$.
\end{proposition}

\begin{corollary}\label{cor:cor4.2}
The cycle length map is continuous.
\end{corollary}

\begin{proposition}\label{pr:prop4.4}
The maps $[0,\tilde{T})\ni\Delta\mapsto\overline{x}_{\Delta}\in\mathbb{R}$ and
$[0,\tilde{T})\ni\Delta\mapsto\underline{x}_{\Delta}\in\mathbb{R}$ are continuous.
\end{proposition}

\section{Computation of the response}\label{sec:pert}

In the Subsections \ref{sec:rising}-\ref{sec:fall-rise} below we keep the parameters $\tau>0,\beta_L>0>-\beta_U, a>0,\sigma\in(0,\tau]$ fixed and  require $-\beta_U+a<0$ as in the preceding section, and study the behaviour of $x^{(\Delta)}$ depending on the onset of the pulse (at $t=\Delta$) and on its termination (at $t=\Delta+\sigma$)  relative to
the zeros and extrema
$$
0<\tilde{z}_1<\tilde{z}_1+\tau=t_{\max}<\tilde{z}_2<\tilde{z}_2+\tau=\tilde{T}<\tilde{z}_3,
$$
of the  periodic solution $\tilde{x}$, on the sign of $x^{(\Delta)}(\Delta+\sigma)$, and on the position of $x^{(\Delta)}(\Delta+\sigma)$ relative to the level $\beta_L$.

The computations that follow in this section can become quite difficult to keep track of, and we therefore use what we hope is a simple and transparent nomenclature to aid the reader in following our progression.  The reader may wish to consult Tables \ref{table:pos-a-pulse-summary}, \ref{table:pos-a-pulse-summary2} and \ref{table:last} as a way of keeping track of the result.

If the pulse starts at $\Delta\in [0,t_{\max})$, where the periodic solution is increasing,  then we say that we are in the rising phase and we use the letter \textbf{R}. If it starts at $\Delta\in [t_{\max},\tilde{T})$, where the periodic solution is decreasing,  then we say that we are in the falling phase and we use the letter \textbf{F}. If $x^{(\Delta)}(t)$ is negative at the beginning of the pulse, i.e., $x^{(\Delta)}(\Delta)<0$, then we use the letter  \textbf{N}  (negative value at $\Delta$)  and otherwise we write \textbf{P} (nonnegative  value at $\Delta$). We can thus say that we are in the subcase \textbf{RN} when we are at rising phase with a negative value at $\Delta$. Therefore the beginning of the pulse can be coded with two letters which gives four subcases: \textbf{RN}, \textbf{RP}, \textbf{FN}, \textbf{FP}. In the same way we can code the end of the pulse, namely if $\Delta+\sigma\in [t_{\max},\tilde{T})$, where the periodic solution is decreasing, then we are in the falling phase and we use the letter \textbf{F}, otherwise we write \textbf{R}. If $x^{(\Delta)}(\Delta+\sigma)<0$, then we use the letter  \textbf{N} and if $x^{(\Delta)}(\Delta+\sigma)\ge 0$ the letter  \textbf{P}.  Here similarly, we can have four subcases and we combine them  together to code each case with four letters. For example the case \textbf{RNRN} corresponds to the rising phase at $\Delta$ and at $\Delta+\sigma$ with negative values of $x^{(\Delta)}$ at $\Delta$ and at $\Delta+\sigma$.

There are three different periods of time that are important:  before the pulse occurs, during the pulse, and after the pulse.  
We can easily write down the values $x^{(\Delta)}(t)$ before the pulse in each phase. If $\Delta\in [0,t_{\max})$ then we have
\begin{equation*}
x^{(\Delta)}(t) =
  \beta_L + (\underline x - \beta_L) e^{-t } =\tilde{x}(t) \quad \text{for}\quad  t \in [ 0, \Delta]
  %\label{e:before}
\end{equation*}
and
the value of $x^{(\Delta)}$ when the pulse turns on is
$$
x^{(\Delta)}(\Delta) =
 \beta_L + (\underline x - \beta_L ) e^{-\Delta}
=\tilde{x}(\Delta).
$$
From the definition of $\underline x$ it follows that $ \beta_L -\underline x=\beta_Le^{\tilde{z}_1}$, which gives the following formula
\begin{equation}\label{e:risingxDelta}
x^{(\Delta)}(\Delta)=\beta_L - \beta_L  e^{\tilde{z}_1-\Delta}\quad \text{for}\quad \Delta\in [0,t_{\max}).
\end{equation}
If
$\Delta\in [t_{\max},\tilde{T})$ then
\begin{equation*}
x^{(\Delta)}(t) =
  -\beta_U + (\overline  x + \beta_U) e^{-(t- t_{\max})}=\tilde{x}(t)  \quad \text{for}\quad t \in [t_{\max},\Delta]
% \label{e:before-falling}
\end{equation*}
and
\[
x^{(\Delta)}(\Delta)= -\beta_U + (\overline  x + \beta_U) e^{-(\Delta- t_{\max})}.
\]
Since $\overline x +\beta_U=\beta_Ue^{\tilde{z}_2-t_{\max}}$, we obtain
\begin{equation}
x^{(\Delta)}(\Delta) = -\beta_U + \beta_U  e^{\tilde{z}_2-\Delta}\quad \text{for}\quad \Delta\in [t_{\max},\tilde{T}).
\label{e:x2+DbU}
\end{equation}

\subsection{A pulse during the rising phase}
\label{sec:rising}
We assume $0\le\Delta<\Delta+\sigma\le t_{\max}$. Then $x^{(\Delta)}(\Delta)$ is given by \eqref{e:risingxDelta}. During the pulse,
\begin{equation}
   x^{(\Delta)}(t) =  (\beta_L + a) + \bigl(x^{(\Delta)}(\Delta) - (\beta_L + a)\bigr) e^{-(t- \Delta)}   \quad \mbox{for} \quad t \in [  \Delta,   \Delta + \sigma], \label{e:during}
\end{equation}
and
after the pulse,
\begin{align}
x^{(\Delta)}(t) & =      \beta_L + (x^{(\Delta)}(\Delta + \sigma) - \beta_L) e^{-\left(t-(  \Delta + \sigma)\right)}\nonumber\\
&  \mbox{for} \quad t>\Delta + \sigma\quad\mbox{as long as}\quad x^{(\Delta)}(t-\tau)<0, \label{e:after}
\end{align}
with
$$
x^{(\Delta)}( \Delta + \sigma) =
\beta_L + a + \bigl(x^{(\Delta)}(\Delta) - (\beta_L + a)\bigr) e^{-\sigma}.
$$
Using \eqref{e:risingxDelta} we have
\begin{equation}\label{e:risingxDeltasigma}
x^{(\Delta)}(\Delta+\sigma)=\beta_L - \beta_L  e^{\tilde{z}_1-\Delta-\sigma}+a(1-e^{-\sigma}).
\end{equation}

\begin{description}
\item[Case $\mathbf{RNR}$] The pulse starts before $\tilde{z}_1=t_{\max}-\tau$, $ 0\le\Delta <\tilde{z}_1$.
\end{description}

Observe first that  $\Delta\in [0,\tilde{z}_1)$ is such that $x^{(\Delta)}(\Delta+\sigma)<0$ if and only if
\[
\beta_L+a(1-e^{-\sigma})< \beta_L  e^{\tilde{z}_1-\Delta-\sigma}
\]
which in view of $\beta_L+a(1-e^{-\sigma})>0$ is equivalent to
\[
e^{\Delta}< \dfrac{\beta_L  e^{\tilde{z}_1-\sigma}}{\beta_L+a(1-e^{-\sigma})}.
\]
Let us define
\begin{equation}\label{d:delta1}
\delta_1=\ln \dfrac{\beta_L  e^{\tilde{z}_1-\sigma}}{\beta_L+a(1-e^{-\sigma})}=\ln\dfrac{\beta_L+\beta_U(1-e^{-\tau})}{\beta_Le^{\sigma}+a(e^{\sigma}-1)}.
\end{equation}
We have
\[
\delta_1=\tilde{z}_1-\sigma-\ln\dfrac{\beta_L+a(1-e^{-\sigma})}{\beta_L}<\tilde{z}_1.
\]
If $\delta_1>0$ we obtain $x^{(\Delta)}(\Delta+\sigma)<0$ for $0\le\Delta<\delta_1$ and $x^{(\Delta)}(\Delta+\sigma)\ge 0$ for $\Delta\in[\delta_1,\tilde{z}_1)$ while if $\delta_1\le 0$ we have $x^{(\Delta)}(\Delta+\sigma)\ge 0$ for all $\Delta\in [0,\tilde{z}_1)$. We consider three subcases.

\begin{description}
\item[Case $\mathbf{RNRN}$] The pulse parameters $(a,\Delta,\sigma)$  are such that $x^{(\Delta)}(t)$ remains negative during the pulse, $x^{(\Delta)}(\Delta+\sigma)<0$. Equivalently, $\delta_1>0$ and
    $\Delta\in[0,\delta_1)= I_{RNRN}$.
\end{description}

\begin{proposition}\label{c:prop1}
If $\Delta\in I_{RNRN}=[0,\delta_1)$ then  $\underline{x}_{\Delta}=\underline{x}$,    $\overline{x}_{\Delta}=\overline{x}$,  and
\begin{equation}
T(\Delta)= \tilde{T}+\ln\left(1-\dfrac{a(e^{\sigma}-1)}{\beta_L}e^{\Delta-\tilde{z}_1}\right)<\tilde{T}.
\label{eqn:prop5.1}
\end{equation}
In particular, the restriction of the map $T$ to $I_{RNRN}$  is strictly decreasing.
\end{proposition}

\begin{figure}[htb]
\centering

\includegraphics{pcr-2.mps}

\caption{A schematic representation of the solution of the DDE when  $0\le x^{(\Delta)}(\Delta+\sigma)\le\beta_L$. The unperturbed periodic solution $\tilde x$ is the solid line  and the solution $x^{(\Delta)}$ is the dashed line.}
\label{fig:bl-a2}
\end{figure}
We now consider the case  \textbf{RNRP} when $x^{(\Delta)}(\Delta+\sigma)\ge 0$. From Figure~\ref{fig:bl-a2} we expect that the first local maximum of $x^{(\Delta)}$ after $t=\Delta$ is achieved before  $t=t_{\max}$ and is not smaller than $\overline{x}$. In the following we prove this.  Also we shall obtain a result about the cycle length $T(\Delta)$ of $x^{(\Delta)}$. However this time we can not conclude that either $T(\Delta) > \tilde{T}$ or $T(\Delta) < \tilde{T}$,  see Figures~\ref{fig:bl-a2} and~\ref{fig:bl-a2b}, respectively.
\begin{figure}[htb]
\centering

\includegraphics{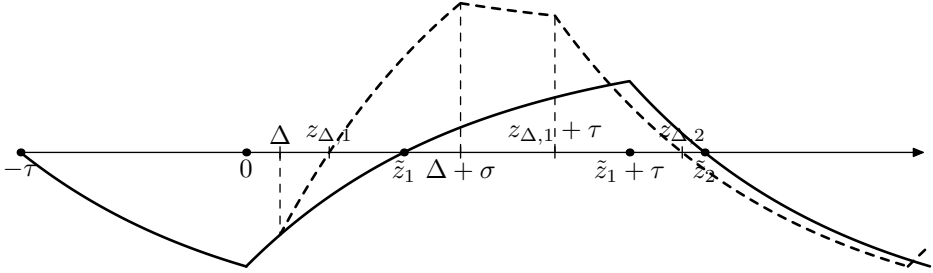}

\caption{A schematic representation of the solution of the DDE when  $x^{(\Delta)}(\Delta+\sigma)>\beta_L$. As usual, the unperturbed periodic solution $\tilde x$ is the solid line  and the solution $x^{(\Delta)}$ is the dashed line.}
\label{fig:bl-a2b}
\end{figure}

From \eqref{e:during} in combination with $x^{(\Delta)}(\Delta)<0<\beta_L+a$ we see that $x^{(\Delta)}$ is strictly increasing on $[\Delta,\Delta+\sigma]$. So by $0\le x^{(\Delta)}(\Delta+\sigma)$ we obtain a first positive zero $z_{\Delta,1}$ of $x^{(\Delta)}$, and
$$
\Delta<z_{\Delta,1}\le\Delta+\sigma.
$$

\begin{proposition}\label{c:prop2} If $\Delta\in [\max\{0,\delta_1\},\tilde{z}_1)$ then $T(\Delta)<\infty$,  $\underline{x}_{\Delta}=\underline{x}$,
$
\overline{x}_{\Delta}\ge\overline{x}
$,
and
\begin{equation}
T(\Delta)=\tilde{T}+\ln\left(1+\dfrac{a(e^{\sigma}-1)}{\beta_U} e^{\Delta-\tilde{z}_2}+\dfrac{a(\beta_L+\beta_U)e^{\tau+\tilde{z}_1-\tilde{z}_2}}{\beta_U(\beta_L+a)}(e^{\Delta-\tilde{z}_1}-1)\right).
\label{eqn:prop5.2}
\end{equation}
Moreover, the maximal value $\overline{x}_{\Delta}$ is given by  $\max\{x^{(\Delta)}(z_{\Delta,1}+\tau),x^{(\Delta)}(\Delta+\sigma)\}$ and is strictly increasing  with respect to $\Delta\in  [\max\{0,\delta_1\},\tilde{z}_1)$.
\end{proposition}

\begin{description}
\item[Case $\mathbf{RPRP}$] The pulse occurs completely in the interval $[\tilde{z}_1, t_{\max}]$. This is equivalent to $\Delta\in[\tilde{z}_1,t_{\max}-\sigma]= I_{RPRP}$.
\end{description}

\begin{proposition}\label{c:prop3}
If   $\Delta\in [\tilde{z}_{1},t_{\max}-\sigma]$ then $T(\Delta)<\infty$, $\underline{x}_{\Delta}=\underline{x}$, $ \overline{x}_{\Delta}> \overline{x}$,
and
\begin{equation}\label{e:TPbetaLB}
T(\Delta)=\tilde{T}+\ln\left(1+\dfrac{a(e^{\sigma}-1)}{\beta_U}e^{
\Delta-\tilde{z}_2}\right)>\tilde{T}.
\end{equation}
Moreover, the map $\overline{x}_{\Delta}$ is strictly increasing  on $I_{RPRP}=[\tilde{z}_{1},t_{\max}-\sigma]$.
\end{proposition}

\begin{remark}\label{rem:rem1}
For $\delta_1>0$ and $\Delta$ close to $\delta_1$ we have $T(\Delta)<\tilde{T}$.
\end{remark}

\subsection{A pulse from the rising phase into the falling phase}
\label{sec:rise-fall}
Here we assume $\Delta\le t_{\max}<\Delta+\sigma$. Then $x^{(\Delta)}(\Delta)$ is still given by \eqref{e:risingxDelta}. Since $\sigma\le \tau$, we must have $\Delta> \tilde{z}_1$, thus $x^{(\Delta)}(\Delta)>0$.
The largest zero of $\tilde{x}$ in $(-\infty,\Delta]$ is $\tilde{z}_1$. Also, with $t_{\max}=\tilde{z}_1+\tau$,
\begin{align*}
 x^{(\Delta)}(t_{\max})&=\beta_L+a+\bigl(x^{(\Delta)}(\Delta)-(\beta_L+a)\bigr)e^{-(
t_{\max}-\Delta)}\\
 & = \beta_L-\beta_Le^{-\tau}+a\bigl(1-e^{-(
t_{\max}-\Delta)}\bigr)\\
 & =  \overline x + a(1-e^{\Delta-
t_{\max}})\ge\overline{x},
\end{align*}
and
we can write
\[
x^{(\Delta)}(t)=-\beta_U+a+\bigl(x^{(\Delta)}(t_{\max})+\beta_U-a\bigr)e^{-(
t-t_{\max})}\quad\text{for}\quad t\in [t_{\max},\Delta+\sigma].
\]
Observe that the function $[t_{\max},\Delta+\sigma)\ni t\mapsto x^{(\Delta)}(t)\in\mathbb{R}$ is decreasing since  $x^{(\Delta)}(t_{\max})+\beta_U-a\ge 0$.

We have
\begin{align*}
x^{(\Delta)}(\Delta+\sigma) & =  -\beta_U+a+\bigl(x^{(\Delta)}(t_{\max})+\beta_U-a\bigr)e^{-(
\Delta+\sigma-t_{\max})}\\
 & =  -\beta_U+a+\bigl(\beta_U e^{\tilde{z}_2-t_{\max}}-ae^{\Delta-
t_{\max}}\bigr)e^{-(
\Delta+\sigma-t_{\max})},
\end{align*}
which gives
\begin{equation}\label{e:xprf}
x^{(\Delta)}(\Delta+\sigma)=-\beta_U+\beta_Ue^{\tilde{z}_2-(
\Delta+\sigma)} +a(1-e^{-\sigma}).
\end{equation}
Using
$$
\tilde{z}_1+\tau<\Delta+\sigma\le\tilde{z}_1+\tau+\sigma<\tilde{z}_2+\tau
$$
we have
$$
x^{(\Delta)}(t)=-\beta_U+\bigl(x^{(\Delta)}(\Delta+\sigma)+\beta_U\bigr)e^{-\left(t-(\Delta+\sigma)\right)}\quad\text{for}\quad \Delta + \sigma \le t\le\tilde{z}_2+\tau
$$
and in particular
\[
x^{(\Delta)}(\Delta+\sigma)=\tilde{x}(\Delta+\sigma)+a(1-e^{-\sigma}),
\]
which yields $x^{(\Delta)}(\Delta+\sigma)>\tilde{x}(\Delta+\sigma)$.

We say that we are in
\begin{description}
\item[Case $\mathbf{RPF}$]
We distinguish between the subcases
\begin{description}
\item[P] $x^{(\Delta)}(\Delta+\sigma)\ge 0$, see Figure~\ref{fig:bl-b---bu-c1---I},
\item[N] $x^{(\Delta)}(\Delta+\sigma)<0$, see Figure~\ref{fig:bl-b---bu-c1---II}.
\end{description}
\end{description}

\begin{figure}[htb]
\centering

\includegraphics{pcr-3.mps}

\caption{A schematic representation of the solution of the DDE for the case $\mathbf{RPFP}$. The unperturbed limit cycle is the solid line while the solution with the pulse is the dashed line.}
\label{fig:bl-b---bu-c1---I}
\end{figure}

\begin{figure}[htb]
\centering

\includegraphics{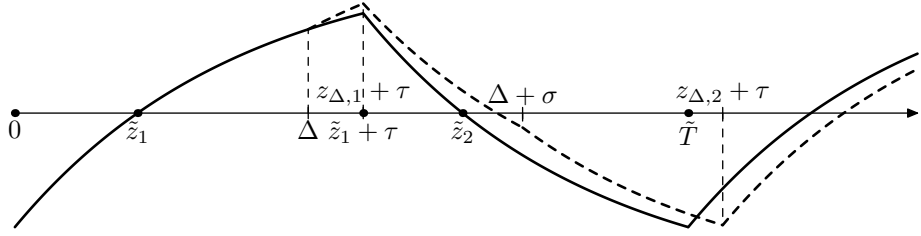}

\caption{A schematic representation of the solution of the DDE for the case $\mathbf{RPFN}$.  The unperturbed limit cycle is the solid line while the solution with the pulse is the dashed line.}
\label{fig:bl-b---bu-c1---II}
\end{figure}

Note that $\Delta\in (t_{\max}-\sigma,t_{\max}]$ is such that $x^{(\Delta)}(\Delta+\sigma)\ge 0$ if and only if
\[
\beta_Ue^{\tilde{z}_2-(
\Delta+\sigma)}\ge \beta_U-a(1-e^{-\sigma}).
\]
Since  $ \beta_U>a\ge a(1-e^{-\sigma})$, we have  $
\beta_U-a(1-e^{-\sigma})>0,
$
which implies that  $x^{(\Delta)}(\Delta+\sigma)\ge 0$ if and only if
\[
e^{\Delta}\le \dfrac{\beta_Ue^{\tilde{z}_2}}{\beta_Ue^{\sigma}-a(e^{\sigma}-1)}.
\]
Let us define $\delta_2$ by
\begin{equation}\label{d:delta2}
\delta_2=\ln\dfrac{\beta_Ue^{\tilde{z}_2-\sigma}}{\beta_U-a(1-e^{-\sigma})}.
\end{equation}
We conclude that the $\Delta$-intervals in the two subcases are of the form
\begin{align*}
I_{RPFP} & =  (t_{\max}-\sigma,t_{\max}]\cap(-\infty,\delta_2]\\
 \text{and} & \\
I_{RPFN} &  =  (t_{\max}-\sigma,t_{\max}]\cap(\delta_2,\infty).
\end{align*}

We have
\begin{equation*}
\delta_2=\tilde{z}_2-\sigma+\ln\dfrac{\beta_U}{\beta_U-a(1-e^{-\sigma})}.
\end{equation*}
Notice that
$
\delta_2>\tilde{z}_2-\sigma,
$
which implies that
$
\delta_2>t_{\max}-\sigma
$.
Consequently, if
\[
\beta_Ue^{\sigma}-a(e^{\sigma}-1)\le \beta_U+\beta_L(1-e^{-\tau}),
\]
which is equivalent to $\delta_2\ge t_{\max}$,
then $x^{(\Delta)}(\Delta+\sigma)\ge 0$ for all $\Delta\in (t_{\max}-\sigma,t_{\max}]$.
If the reverse inequality
\begin{equation}\label{e:betaLBbetaUII}
 \beta_U+\beta_L(1-e^{-\tau})<\beta_Ue^{\sigma}-a(e^{\sigma}-1)
\end{equation}
holds then $\delta_2<t_{\max}$ and $\delta_2$ is the maximal value of $\Delta\in (t_{\max}-\sigma,t_{\max}]$ such that
$x^{(\Delta)}(t)$ changes the sign at $\Delta+\sigma$, from  a positive value to a negative value, $x^{(\Delta)}(\Delta+\sigma)\ge 0$ for all $\Delta\in (t_{\max}-\sigma,\delta_2]$ and $x^{(\Delta)}(\Delta+\sigma)< 0$ for all $\Delta\in(\delta_2,t_{\max}]$.
Inequality \eqref{e:betaLBbetaUII} can be rewritten as
\[
\beta_L(1-e^{-\tau})<(\beta_U-a)(e^{\sigma}-1).
\]

\begin{proposition} \label{c:prop4}
If $\Delta\in I_{RPFP}$  then $T(\Delta)<\infty$, $\underline{x}_{\Delta}=\underline{x}$, $\overline{x}_{\Delta}>\overline{x}$,
and $T(\Delta)$ is  given by formula \eqref{e:TPbetaLB} as in Proposition \ref{c:prop3}. The map $I_{RPFP} \ni\Delta\mapsto \overline{x}_{\Delta}\in\mathbb{R}$ is strictly decreasing and $I_{RPFP} \ni\Delta\mapsto T(\Delta)\in\mathbb{R}$ is strictly increasing.
\end{proposition}

\begin{proposition}\label{c:prop5}
If $\Delta\in I_{RPFN}$  then $T(\Delta)<\infty$,  $\underline{x}_{\Delta}>\underline{x}$,
\[
\overline{x}_{\Delta}=\overline{x}+a\bigl(1-e^{\Delta-t_{\max}}\bigr)\ge\overline{x},
\]
and
\begin{equation}
T(\Delta)=\tilde{T}+\ln\left(1-\dfrac{a(e^{\sigma}-1)}{\beta_L}e^{\Delta-\tilde{z}_1-\tilde{T}}-\dfrac{a(\beta_L+\beta_U)e^{-\tilde{z}_1}}
{\beta_L(\beta_U-a)}(e^{\Delta-\tilde{z}_2}-1)\right).
\label{eqn:prop5.5}
\end{equation}
Moreover, the map $I_{RPFN} \ni\Delta\mapsto \underline{x}_{\Delta}\in\mathbb{R}$ is strictly increasing  and  the map  $I_{RPFN} \ni\Delta\mapsto T(\Delta)\in\mathbb{R}$ is strictly decreasing.
\end{proposition}

\begin{remark}\label{rem:rem2}
For   $\Delta$ close to $\delta_2$ we have $T(\Delta)>\tilde{T}$.
\end{remark}

\subsection{A pulse during the falling phase}
\label{sec:fall}
Suppose now that $x^{(\Delta)}(\Delta)$ is given by~\eqref{e:x2+DbU} for $\Delta\in [t_{\max},\tilde{T})$. We have
\begin{equation*}
   x^{(\Delta)}(t) =  -\beta_U + a + \bigl(x^{(\Delta)}(\Delta) + \beta_U-a\bigr) e^{-(t-\Delta)}
\end{equation*}
for $\Delta\le t\le\Delta+\sigma$ as long as $x^{(\Delta)}(t-\tau)>0$.

We assume that  $t_{\max}\le \Delta<\Delta+\sigma\le\tilde{z}_2+\tau=\tilde{T}$.
Then the value of $x^{(\Delta)}$ at the end of the pulse is
\begin{equation*}
x^{(\Delta)}(\Delta + \sigma) = -\beta_U + a + \bigl(x^{(\Delta)}(\Delta) +\beta_U - a\bigr) e^{- \sigma }
\end{equation*}
and, by \eqref{e:x2+DbU}, it is  the same as in \eqref{e:xprf}.
 We proceed in steps as before.

\begin{description}
\item[Case $\mathbf{FPF}$] The pulse starts in the interval $[t_{\max},\tilde{z}_2]$. We distinguish between two subcases.
\begin{description}
\item[$\mathbf{P}$] The  pulse parameters $(a,\Delta, \sigma )$ are such that  $x^{(\Delta)}(t)$ remains nonnegative during the pulse, and changes the sign after the pulse, $x^{(\Delta)}(\Delta+\sigma)\ge0$.  
\item[$\mathbf{N}$] The pulse parameters $(a,\Delta,\sigma)$ are such that $x^{(\Delta)}$ changes the sign from positive to negative during the pulse (see Figure \ref{fig:bu-c2}), $t_{\max}\le\Delta\le \tilde{z}_2$ and $x^{(\Delta)}(\Delta+\sigma)<0$.  
\end{description}
\end{description}

 Since $\Delta\in[t_{\max},\tilde{z}_2]$  it follows from \eqref{e:xprf} that $x^{(\Delta)}(\Delta+\sigma)\ge0$ if and only if
\[
-\beta_U + \beta_U  e^{\tilde{z}_2-\Delta-\sigma}+a (1- e^{- \sigma })\ge 0,
\]
or equivalently,
\[
\beta_U  e^{\tilde{z}_2-\Delta}\ge \beta_U e^{\sigma} -a ( e^{ \sigma }-1).
\]
The corresponding $\Delta$-intervals are of the form
\[
I_{FPFP}=[t_{\max},\tilde{z}_2]\cap(-\infty,\delta_2]
\]
and
\[
I_{FPFN}=[t_{\max},\tilde{z}_2]\cap(\delta_2,\infty).
\]

\begin{proposition}\label{c:prop6} If $\Delta\in I_{FPFP}$  then $T(\Delta)<\infty$, $\overline{x}_{\Delta}=\overline{x}$, $\underline{x}_{\Delta}=\underline{x}$, and  $T(\Delta)$ is given by formula \eqref{e:TPbetaLB}  in Proposition \ref{c:prop3}.
\end{proposition}

We  turn to case \textbf{FPFN}. From Figure \ref{fig:bu-c2} we expect $\tilde{z}_2\le z_{\Delta,2}$ and $\underline{x}_{\Delta}=x^{(\Delta)}(z_{\Delta,2}+\tau)>\tilde{x}(\tilde{z}_2+\tau)=\underline{x}$, that is, the minimum value of $x^{(\Delta)}$ is above the minimum value of $\tilde{x}$.
\begin{figure}[htb]
\centering

\includegraphics{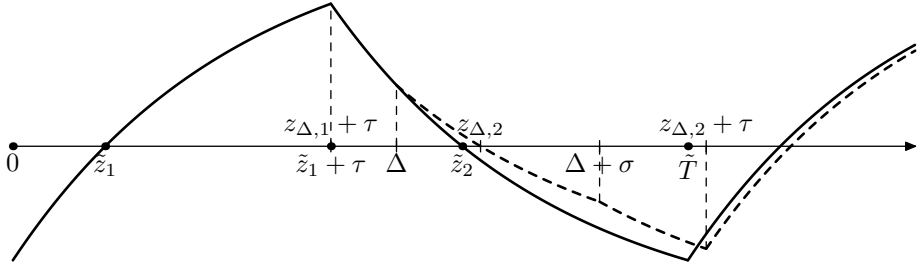}

\caption{A schematic representation of the solution of the DDE for case \textbf{FPFN}.  The unperturbed periodic solution $\tilde x$ is the solid line while the solution $x^{(\Delta)}$ with the pulse is the dashed line.}
\label{fig:bu-c2}
\end{figure}

\begin{proposition}\label{c:prop7} If $\Delta\in  I_{FPFN}$ then $T(\Delta)<\infty$, $\overline{x}_{\Delta}=\overline{x}$, $\underline{x}_{\Delta}>\underline{x}$, and both $\underline{x}_{\Delta}$ and $T(\Delta)$ are as in Proposition \ref{c:prop5}.
\end{proposition}

Assume now that $\Delta\in (\tilde{z}_2,\tilde{z}_2+\tau-\sigma)$. From $x^{(\Delta)}(\Delta)=-\beta_U+(0+\beta_U)e^{-(\Delta-\tilde{z}_2)}$ and
$$
x^{(\Delta)}(\Delta+\sigma)=-\beta_U+a+\bigl(x^{(\Delta)}(\Delta)+\beta_U-a\bigr)e^{-\sigma}
$$
we have $x^{(\Delta)}(\Delta)<0$ and $x^{(\Delta)}(\Delta+\sigma)<0$. Thus we consider the case
\begin{description}
\item[Case $\mathbf{FNFN}$] The pulse parameters $(a,\Delta,\sigma)$ are such that the pulse begins after $x^{(\Delta)}$ changes the sign  from positive to negative and ends before the time $\tilde{T}=\tilde{z}_2+\tau$, $\tilde{z}_2<\Delta$ and $\Delta+\sigma< \tilde{z}_2+\tau$, see Figure \ref{fig:bu-d}.
 \end{description}

\begin{figure}[htb]
\centering

\includegraphics{pcr-7.mps}

\caption{A schematic representation of the solution for case $\mathbf{FNFN}$.  The unperturbed periodic solution $\tilde x$ is the solid line while the solution $x^{(\Delta)}$ with the pulse is the dashed line.}
\label{fig:bu-d}
\end{figure}

 From Figure \ref{fig:bu-d} we expect  that the minimum value of $x^{(\Delta)}$ in $[\tilde{z}_2,\tilde{z}_2+\tau]$  is above the minimum value $\underline{x}$ of $\tilde{x}$, and  the cycle length $T(\Delta)$ is below the minimal period $\tilde{T}$ of $\tilde{x}$.
Observe that the function $[\Delta,\Delta+\sigma]\ni t\mapsto x^{(\Delta)}(t)\in\mathbb{R}$ is increasing if and only if
\[
x^{(\Delta)}(\Delta)+\beta_U-a<0,
\]
which is equivalent to $\beta_Ue^{\tilde{z}_2}<ae^{\Delta}$.
Now, if $t\in[\Delta+\sigma,\tilde{T}]$
then $\tilde{z}_1\le t-\tau\le \tilde{z}_2$. Thus
\[
x^{(\Delta)}(t)  =  -\beta_U+\bigl(x^{(\Delta)}(\Delta+\sigma)+\beta_U\bigr)
e^{-(t-(\Delta+\sigma))}
\]
and
we have
\[
x^{(\Delta)}(\Delta+\sigma)+\beta_U=\beta_Ue^{\tilde{z}_2-(
\Delta+\sigma)} +a(1-e^{-\sigma}),
\]
which is always positive. Hence, the function $t\mapsto x^{(\Delta)}(t)$ is strictly decreasing on the interval $[\Delta+\sigma,\tilde{T}]$ and we have
\[
x^{(\Delta)}(\tilde{T})  =  -\beta_U+\bigl(x^{(\Delta)}(\Delta+\sigma)+\beta_U\bigr)
e^{-\left(\tilde{T}-(\Delta+\sigma)\right)}
\]
which becomes
\begin{align*}
x^{(\Delta)}(\tilde{T})&=-\beta_U+\bigl(\beta_Ue^{\tilde{z}_2-(
\Delta+\sigma)} +a(1-e^{-\sigma})\bigr)
e^{-\left(\tilde{T}-(\Delta+\sigma)\right)}\\
&=-\beta_U+\beta_Ue^{-\tau}+a(e^{\sigma}-1)
e^{\Delta-\tilde{T}}.
\end{align*}

\begin{figure}[htb]

\centering
\includegraphics{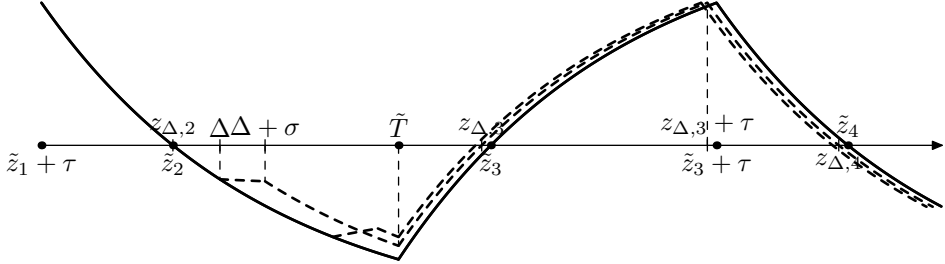}
\caption{Graphs of two solutions for case  $\mathbf{FNFN}$, where the parameters are $\tau=1$, $\sigma=0.2$, $\beta_U=0.8$, $\beta_L=1$, $a=0.6$, and $\Delta\in \{2.2,2.7\}$, showing that the function $I_{FNFN}\ni\Delta\mapsto \underline{x}_{\Delta}$ is not given by one formula.}
\end{figure}

Define $\overline{\delta}\in\mathbb{R}$ by
$$
e^{\tilde{T}-\overline{\delta}}=\dfrac{\beta_U+\sqrt{\beta_U ^2+4a\beta_U(e^{\sigma}-1)e^{\tau}}}{2\beta_U}.
$$

\begin{proposition}\label{c:prop8} If  $\Delta\in I_{FNFN}$ then $T(\Delta)<\infty$,  $\overline{x}_{\Delta}=\overline{x}$,
$$
\underline{x}_{\Delta}=\min\{\tilde{x}(\Delta),x^{(\Delta)}(\tilde{T})\}>\underline{x},
$$
and
\begin{equation}
T(\Delta)=\tilde{T}+\ln\left(1-\dfrac{a(e^{\sigma}-1)}{\beta_L}e^{\Delta-\tilde{z}_1-\tilde{T}}\right)<\tilde{T}.
\label{e:prop5.9}
\end{equation}
In case $I_{FNFN}\cap(-\infty,\overline{\delta})\neq\emptyset$ the map
$$
I_{FNFN}\cap(-\infty,\overline{\delta})\ni\Delta\mapsto\underline{x}_{\Delta}\in\mathbb{R}
$$
is strictly increasing while in case $I_{FNFN}\cap(\overline{\delta},\infty)\neq\emptyset$ the map
$$
I_{FNFN}\cap(\overline{\delta},\infty)\ni\Delta\mapsto\underline{x}_{\Delta}\in\mathbb{R}
$$
is strictly decreasing.
\end{proposition}

\begin{remark}\label{rem:rem3}
If $\beta_U e^{\sigma - \tau } <  a$ then $\overline{\delta}\in I_{FNFN}=(\tilde{z}_2,\tilde{T}-\sigma)$.  If $\beta_U e^{\sigma - \tau} \ge a$ then $\overline{\delta}\ge\tilde{T}-\sigma$.
\end{remark}

\begin{remark} 
The following relations are equivalent:
\begin{align*}
a & < \beta_U,\\
a(e^{\sigma}-1) & <  \beta_U(e^{\sigma}-1),\\
\beta_U & <  \bigl(\beta_U-a(1-e^{-\sigma})\bigr)e^{\sigma},\\
\ln\dfrac{\beta_U}{\beta_U-a(1-e^{-\sigma})} & < \sigma,\\
\delta_2 & <  \tilde{z}_2.
\end{align*}
Observe that we have
$\tilde{z}_2\le \delta_2< \tilde{T}-\sigma$ if and only if
\begin{equation}\label{c:opo}
\beta_U\le a\quad\text{and}\quad a(1-e^{-\sigma})<\beta_U(1-e^{-\tau}),
\end{equation}
which is excluded by our standing hypothesis $a<\beta_U$.

Let us briefly address what might happen if  \eqref{c:opo} holds.
In that case if $\Delta\in (\tilde{z}_2,\tilde{T}-\sigma)$ then $x^{(\Delta)}(\Delta+\sigma)\ge 0$ if and only if $\Delta\le \delta_2$. We  have $I_{FNFN}=(\delta_2,\tilde{T}-\sigma)$ and Proposition \ref{c:prop8} remains true. The case \textbf{FNFP} is possible with $I_{FNFP}=(\tilde{z}_2,\delta_2]$.  Figure  \ref{fig:FNFP} and in particular Figure \ref{fig:FNFPL} indicate that in the interval $I_{FNFP}$ the cycle length may not be finite everywhere, due to a higher oscillation frequency of the solution on $[\Delta+\sigma,\infty)$.
\end{remark}

\begin{figure}[htb]

\centering
\includegraphics{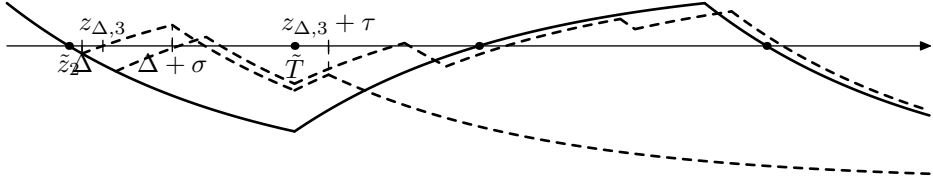}
\caption{Graphs of two solutions for case  $\mathbf{FNFP}$, where the parameters are $\tau=1$, $\sigma=0.4$, $\beta_U=0.6$, $\beta_L=0.3$, $a=0.95$, and $\Delta\in \{2.15,2.3\}$, showing that  $T(\Delta)$ might not exists.  }\label{fig:FNFP}
\end{figure}

\begin{figure}[htb]

\centering
\includegraphics{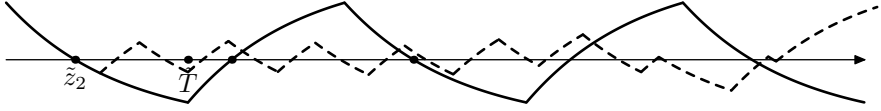}
\caption{A graph of one solution for case  $\mathbf{FNFP}$, where the parameters are $\tau=1$, $\sigma=0.35$, $\beta_U=0.3$, $\beta_L=0.4$, $a=0.7$, and $\Delta=2.22$, showing that $T(\Delta)$ is not given by one formula.}\label{fig:FNFPL}
\end{figure}

\subsection{A pulse from the falling phase into the rising phase}
\label{sec:fall-rise}
Here we suppose that $\tilde{z}_2<\Delta<\tilde{T}\le \Delta+\sigma$ and that $x^{(\Delta)}(\Delta)$ is given by~\eqref{e:x2+DbU}  for $\Delta\in [\tilde{T}-\sigma,\tilde{T})$. We have
\begin{equation*}
   x^{(\Delta)}(t) =  -\beta_U + a + \bigl(\tilde{x}(\Delta) + \beta_U-a\bigr) e^{-(t-\Delta)}   \quad \mbox{for} \quad t \in ( \Delta, \tilde{T}]
\end{equation*}
and the function  $[\Delta, \tilde{T}]\ni t\mapsto x^{(\Delta)}(t)\in\mathbb{R}$ is either strictly increasing, or decreasing.
We have
\begin{align*}
 x^{(\Delta)}(\tilde{T})&=-\beta_U+a+\bigl(x^{(\Delta)}(\Delta)+\beta_U-a\bigr)e^{-(
\tilde{T}-\Delta)}\\
 & = -\beta_U+a+(-\beta_U+\beta_U e^{-(\Delta-\tilde{z}_2)}+\beta_U-a)e^{-(
\tilde{T}-\Delta)}\\
 & = - \beta_U+\beta_U
e^{-\tau}+a\bigl(1-e^{-(\tilde{T}-\Delta)}\bigr)\\
&=\underline{x}+a\bigl(1-e^{\Delta-\tilde{T}}\bigr)>\underline{x},
\end{align*}
which shows that the map
$$
[\tilde{T}-\sigma,\tilde{T})\ni\Delta\mapsto x^{(\Delta)}(\tilde{T})\in\mathbb{R}
$$
is strictly decreasing.

For $t\in[\tilde{T},\Delta+\sigma]$ we have
\[
x^{(\Delta)}(t)  =  \beta_L+a+\bigl(x^{(\Delta)}(\tilde{T})-(\beta_L+a)\bigr)e^{-(
t-\tilde{T})}.
\]
Since
\[
x^{(\Delta)}(\tilde{T})-(\beta_L+a)=-(\beta_L-\underline{x})- ae^{\Delta-\tilde{T}}<0,
\]
the function $[\tilde{T},\Delta+\sigma]\ni t\mapsto x^{(\Delta)}(t)\in\mathbb{R}$ is strictly increasing.
Also
\begin{align*}
x^{(\Delta)}(\Delta+\sigma) & =  \beta_L+a+\bigl(x^{(\Delta)}(\tilde{T})-(\beta_L+a)\bigr)e^{-(
\Delta+\sigma-\tilde{T})}\\
& =  \beta_L+a+\bigl(\underline{x}-a\,e^{\Delta-\tilde{T}
}-\beta_L\bigr)e^{-(
\Delta+\sigma-\tilde{T})}\\
& = \beta_L+(\underline{x}-\beta_L)e^{-(
\Delta+\sigma-\tilde{T})} +a(1-e^{-\sigma})\quad(>\tilde{x}(\Delta+\sigma))\\
& = \beta_L-\beta_Le^{\tilde{z}_1-(
\Delta+\sigma-\tilde{T})}+a(1-e^{-\sigma}).
\end{align*}

\begin{description}
\item[Case $\mathbf{FNR}$] We must distinguish between the two subcases
\begin{description}
\item[N] $x^{(\Delta)}(\Delta+\sigma)< 0$, see Figure \ref{fig:bu-d---bl-a1---III},
\item[P]  $x^{(\Delta)}(\Delta+\sigma)\ge 0$, see Figure \ref{fig:bu-d---bl-a1---IV}.
\end{description}
\end{description}

\begin{figure}[htb]
\centering

\includegraphics{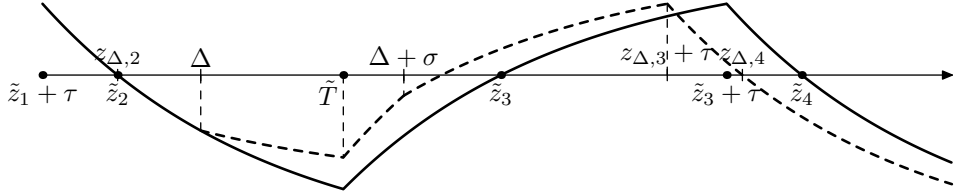}

\caption{A schematic representation of the solution of the DDE for the case $\mathbf{FNRN}$.  The unperturbed limit cycle is the solid line  while the solution with the pulse is the dashed line.}
\label{fig:bu-d---bl-a1---III}
\end{figure}

\begin{figure}[htb]
\centering

\includegraphics{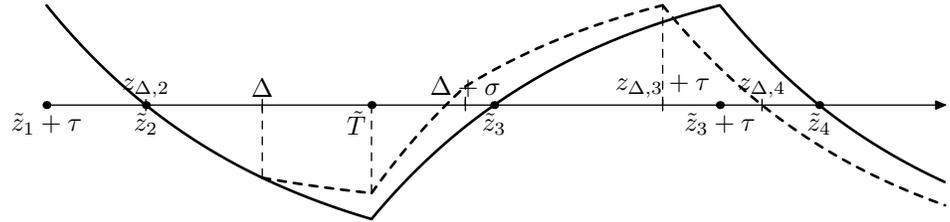}

\caption{A schematic representation of the solution of the DDE for the  case $\mathbf{FNRP}$. }
\label{fig:bu-d---bl-a1---IV}
\end{figure}

We have $\tilde{T}-\sigma\le \Delta<\tilde{T}$, and the condition for subcase \textbf{FNRN}, namely,
$$
0> x^{(\Delta)}(\Delta+\sigma)=\beta_L-\beta_Le^{\tilde{z}_1-(
\Delta+\sigma-\tilde{T})}+a(1-e^{-\sigma}),
$$
is equivalent to
$$
\beta_L+a(1-e^{-\sigma})< \beta_Le^{\tilde{z}_1-(
\Delta+\sigma-\tilde{T})},
$$
or
$$
e^{\Delta}< e^{\tilde{T}}\dfrac{\beta_L e^{\tilde{z}_1}}{\beta_L e^{\sigma}+a(e^{\sigma}-1)}=e^{\tilde{T}+\delta_1}.
$$
Thus we are in case \textbf{FNRN} if and only if
\[
\Delta  \in  [\tilde{T}-\sigma,\tilde{T})\cap(-\infty,\tilde{T}+\delta_1)= I_{FNRN}.
\]
and we are in case \textbf{FNRP} if and only if
\[
\Delta  \in [\tilde{T}-\sigma,\tilde{T})\cap [\tilde{T}+\delta_1,\infty ).
\]

\begin{proposition}\label{c:prop9} If $\Delta\in I_{FNRN}$ then $T(\Delta)<\infty$,
$\overline{x}_{\Delta}=\overline{x}$, $\underline{x}_{\Delta}
>\underline{x}$, and
$T(\Delta)$ is as in \eqref{e:prop5.9} of Proposition \ref{c:prop8}. The map $I_{FNRN}\ni\Delta\mapsto\underline{x}_{\Delta}\in\mathbb{R}$ is strictly decreasing.
\end{proposition}

The interval $I_{FNRN}$ in Proposition \ref{c:prop9}  is nonempty if and only if $\delta_1>-\sigma$, which is always the case. To see this observe that
\[
e^{\delta_1}=e^{-\sigma}\dfrac{\beta_L -\underline{x}}{\beta_L+a(1-e^{-\sigma})}>e^{-\sigma}
\]
since
\[
-\underline{x}=\beta_U(1-e^{-\tau})>a(1-e^{-\sigma}).
\]
We next proceed to  the subcase $\mathbf{FNRP}$. Note that the interval $I_{FNRP}$  is empty if and only if $\delta_1\ge 0$.

\begin{proposition}\label{c:prop10} If $\Delta\in I_{FNRP}$ then $T(\Delta)<\infty$, $\overline{x}_{\Delta}>\overline{x}$, $\underline{x}_{\Delta}>\underline{x}$, and
\begin{equation}
T(\Delta)=\tilde{T}+\ln\left(1+\dfrac{a(e^\sigma-1)}{\beta_U}e^{\Delta-\tilde{z}_2-\tilde{T}}+
\dfrac{a(\beta_L+\beta_U)e^{\tau+\tilde{z}_1-\tilde{z}_2}}{\beta_U(\beta_L+a)}\left(e^{\Delta-\tilde{z}_1-\tilde{T}}-1\right)\right).
\label{e:prop11}
\end{equation}
The map $I_{FNRP}\ni\Delta\mapsto\underline{x}_{\Delta}\in\mathbb{R}$ is strictly decreasing
as is the map
$I_{FNRP}\ni\Delta\mapsto \overline{x}_{\Delta}\in\mathbb {R}$.
\end{proposition}

\section{The cycle length map}\label{sec:prop-reset}

The behaviour of the cycle length map $[0,\tilde{T})\ni\Delta\mapsto T(\Delta)\in\mathbb{R}$, illustrated in Figures~\ref{fig:tp-1}--\ref{fig:tp-2}, is different in each of the (sub-) cases discussed in Sections \ref{sec:rising}-\ref{sec:fall-rise}. Each of these cases corresponds to $\Delta$ varying in one of the subintervals  $I_{RNRN},\ldots,I_{FNRP}$ of $[0,\tilde{T})$.   If $\Delta$ increases from $0$ to $\tilde{T}$ then it travels through %one sequence after another through
those subintervals
which are not empty for the given parameter vector $(\tau,\beta_L,\beta_U,a,\sigma)$ (with
$0<\tau,-\beta_U<0<\beta_L,0<a<\beta_U,0<\sigma\le\tau)$.
In other words, for each parameter vector we have a finite sequence of non-empty subintervals, ordered by, say, their left endpoints, whose union is $[0,\tilde{T})$. Below we describe the possible scenarios, in terms of sequences of cases and subcases. We also provide tables summarizing the behavior of the minimum, maximum and the cycle length.

\begin{figure}[htb]
\centering
\includegraphics{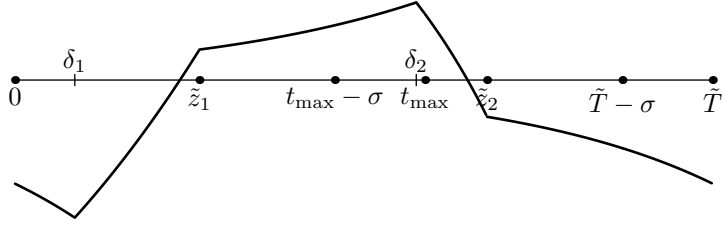}
\caption{A graph of the function $[0,\tilde{T}]\ni\Delta\mapsto T(\Delta)$. The straight line represents the graph of $[0,\tilde{T}]\ni\Delta\mapsto \tilde{T}$. We indicated with dots the values of the boundaries of all cases, and with lines the boundaries of subcases. Here the parameters are $\tau=1$, $\beta_U=0.8$, $\beta_L=0.4$,  and
$\sigma=0.4$, $a=0.2$, so that $\delta_1>0$ and $\delta_2<t_{\max}$. The figure corresponds to the following sequence of subcases: $\mathbf{RNRN}$, $\mathbf{RNRP}$, $\mathbf{RPRP}$, $\mathbf{RPFP}$, $\mathbf{RPFN}$, $\mathbf{FPFN}$, $\mathbf{FNFN}$, $\mathbf{FNRN}$.   }
\label{fig:tp-1}
\end{figure}

\begin{figure}[htb]
\centering
\includegraphics{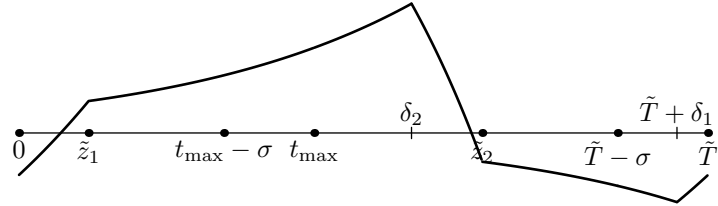}
\caption{As in Figure~\ref{fig:tp-1}, but with $\beta_L=1.4$, which gives $\delta_1\in (-\sigma,0)$, $\delta_2\in (t_{\max},\tilde{z}_2)$,  and  the following sequence of subcases: $\mathbf{RNRP}$, $\mathbf{RPRP}$, $\mathbf{RPFP}$, $\mathbf{FPFP}$, $\mathbf{FPFN}$, $\mathbf{FNFN}$, $\mathbf{FNRN}$, $\mathbf{FNRP}$.}
\label{fig:tp-2}
\end{figure}

Before doing so it may be convenient to collect a few facts about the quantities  $\delta_1,\delta_2$ which, together with $\tilde{z}_1,\tilde{z}_2,t_{\max},\sigma,\tilde{T}$, determine the intervals $I_{RNRN},\ldots,I_{FNRP}$. Recall that $\delta_1$ was defined by \eqref{d:delta1} and that $\delta_2$ was defined by \eqref{d:delta2}. From $\beta_U>a$ and $\tau\ge \sigma$ we have $\delta_1>-\sigma$ and $\delta_2<\tilde{z}_2$.
The condition  $\delta_1>0$ is equivalent to
\begin{equation*}
 (\beta_L+a)(e^{\sigma}-1)<\beta_U(1-e^{-\tau}),
\end{equation*}
and the case when $\delta_1\in (-\sigma,0]$   is described by
\begin{equation*}
a(1-e^{-\sigma})< \beta_U(1-e^{-\tau})\le (\beta_L+a)(e^{\sigma}-1).
\end{equation*}
Next,  the condition  $\delta_2< t_{\max}$ is equivalent to
\begin{equation*}
\beta_L(1-e^{-\tau})< (\beta_U-a)(e^{\sigma}-1)
\end{equation*}
while the condition $t_{\max}\le \delta_2<\tilde{z}_2$  is equivalent to
\[
0< (\beta_U-a)(e^{\sigma}-1)\le \beta_L(1-e^{-\tau}).
\]

Since $\delta_1<\tilde{z}_1$  and  $\delta_2>t_{\max}-\sigma$, we conclude that the intervals
\[
I_{RNRP}=[\min\{0,\delta_1\},\tilde{z}_1),\quad I_{RPFP}=(t_{\max}-\sigma,\min\{t_{\max},\delta_2\}]
\]
are always nonempty.
We have $I_{RPRP}=[\tilde{z}_1,t_{\max}-\sigma]$, $t_{\max}=\tilde z_1+\tau$, and  $\tau\ge \sigma$, thus each sequence of cases contains  %{\color{red}The following needs to be cleaned up and clarified for the reader.}
$$
\mathbf{RNRP}, \quad \mathbf{RPRP},\quad\text{and}\quad\mathbf{RPFP}.
$$
\begin{table}
 \begin{tabular}{|l|c|c|c|c|c|c|c|}
   \hline
   & \multicolumn{7}{|c|}{Pulse location when $\delta_2< t_{\max}$} \\
  \hline \hline
      & $\mathbf{RNRP}$ & $\mathbf{RPRP}$ & $\mathbf{RPFP}$ &$\mathbf{RPFN}$ & $\mathbf{FPFN}$ & $\mathbf{FNFN}$& $\mathbf{FNRN}$\\
   \hline
   $\underline x_\Delta$  & U & U & U & $\uparrow$  $\nearrow$ & $\uparrow$  $\nearrow$ & $\uparrow$ & $\uparrow$ $\searrow$\\ \hline
   $\overline x_\Delta$ & $\uparrow$ $\nearrow$ & $\uparrow$ $\nearrow$& $\uparrow$ $\searrow$ &$\uparrow$ $\searrow$  &U& U & U \\ \hline
   $T(\Delta)$ & $\nearrow$  & $\uparrow$ $\nearrow$ & $\uparrow$ $\nearrow$ & $\searrow$ & $\searrow$ &$\downarrow$ $\searrow$ & $\downarrow$ $\searrow$ \\
  \hline
 \end{tabular}
 \caption{Summary of the effects of a positive pulse $(a > 0, \Delta, \sigma)$ at different times during the limit cycle on the limit cycle minimum ($\underline x$), maximum ($\overline x$), and period ($\tilde{T}$). ``U" denotes unchanged, while $\uparrow$ means increased, $\downarrow$ means decreased, $\nearrow$ means that the value as a function of $\Delta$ is increasing, and $\searrow$ means that the values as a function of $\Delta$ is decreasing.} \label{table:pos-a-pulse-summary}
\end{table}

For $\delta_1>0$ each sequence starts with the case $\mathbf{RNRN}$ and ends with the case $\mathbf{FNRN}$.

For $\delta_1=0$ each sequence starts with the case $\mathbf{RNRP}$ and ends with the case $\mathbf{FNRN}$.

For $\delta_1<0$ each sequence starts with the case $\mathbf{RNRP}$ and ends with the case $\mathbf{FNRP}$.

If $\delta_1>0$ and $0\le\Delta<\delta_1$ then we are in case $\mathbf{RNRN}$.
If $\Delta$ grows from  $\max\{0,\delta_1\}$ to
$\min\{t_{\max},\delta_2\}$  then  we have the  sequence of cases:
$$
\mathbf{RNRP}, \quad \mathbf{RPRP},\quad\text{and}\quad\mathbf{RPFP}.
$$
If $\Delta$ grows from
$\min\{t_{\max},\delta_2\}$ to $\min\{0,\delta_1\}+\tilde{T}$
and if $\delta_2<t_{\max}$ then we obtain the subsequent cases
\[
\mathbf{RPFN},\quad\mathbf{FPFN},\quad\mathbf{FNFN},\quad\mathbf{FNRN}.
\]
In case  $\delta_2=t_{\max}$ we obtain
\[
 \mathbf{FPFP},\quad \mathbf{FPFN},\quad\mathbf{FNFN},\quad \mathbf{FNRN}.
\]
The same sequence results in case $t_{\max}<\delta_2$. So  we have two scenarios for $\Delta$ beyond the interval $I_{RPFP}$ and below $\min\{0,\delta_1\}+\tilde{T}$, which is the endpoint $\tilde{T}$ of the domain of the cycle length map if $\delta_1\ge0$, while for $\delta_1<0$ the sequence of cases is completed by
$$
\mathbf{FNRP}.
$$

\begin{table}
 \begin{tabular}{|l|c|c|c|c|c|c|c|c|}
   \hline
   & \multicolumn{8}{|c|}{Pulse location when $\delta_2\ge t_{\max}$} \\
  \hline \hline
        & $\mathbf{RNRP}$ & $\mathbf{RPRP}$ & $\mathbf{RPFP}$ &$\mathbf{RPFN}$ & $\mathbf{FPFP}$& $\mathbf{FPFN}$ & $\mathbf{FNFN}$& $\mathbf{FNRN}$\\
   \hline
   $\underline x_\Delta$  & U & U & U & $\uparrow$  $\nearrow$ & U & $\uparrow$  $\nearrow$& $\uparrow$ & $\uparrow$ $\searrow$\\ \hline
   $\overline x_\Delta$ & $\uparrow$ $\nearrow$ & $\uparrow$ $\nearrow$& $\uparrow$ $\searrow$ &$\uparrow$ $\searrow$  &U& U& U & U \\ \hline
   $T(\Delta)$  & $\nearrow$  & $\uparrow$ $\nearrow$ & $\uparrow$ $\nearrow$ & $\searrow$ & $\uparrow$ $\nearrow$ &$\searrow$ & $\downarrow$ $\searrow$ & $\downarrow$ $\searrow$ \\
  \hline
 \end{tabular}
  \caption{Summary of the effects of a positive pulse when  $t_{\max}\le \delta_2$.} \label{table:pos-a-pulse-summary2}
\end{table}

\begin{table}[htb]
\centering
\begin{tabular}{|l|c|}
\hline
& $\delta_1>0$\\
\hline
 \hline
&  $\mathbf{RNRN}$\\ \hline
 $\underline x_\Delta$  & U\\ \hline
$\overline x_\Delta$& U \\ \hline
 $T(\Delta)$ & $\downarrow$  $\searrow$ \\ \hline
\end{tabular} \quad
\begin{tabular}{|l|c|}
\hline
& $\delta_1< 0$\\
\hline
 \hline
& $\mathbf{FNRP}$\\ \hline
 $\underline x_\Delta$  & $\uparrow$ $\nearrow$\\ \hline
$\overline x_\Delta$& $\uparrow$ $\nearrow$ \\ \hline
 $T(\Delta)$ &   $\nearrow$ \\ \hline
\end{tabular}
\caption{Summary of the effects of a positive pulse in the remaining cases}\label{table:last}
\end{table}

\section{A therapy plan}\label{sec:therapy}

We first describe the concept of the therapy plan,  in case the evolution of the number of cells in the bloodstream (of a patient, without medical treatment)  is governed by Eq.~\eqref{e:pl-sys}, with the production function $f$ given by~\eqref{e:pcnl}, and $b_U<\gamma\theta<b_L$. For convenience we shall work not with the original variables but with the transformed quantities from Section 2, namely, Eq.~\eqref{e:pl-sys-dimen1}
with $f$ satisfying Eq.~\eqref{e:pcnl1} for $-\beta_U<0<\beta_L$. Then the variable $x(t)$ still represents the number of blood cells (of a certain type) in a patient, at time $t$.

The reader will find it helpful to consult Figure \ref{fig:therapy} when following the argument below.

\begin{figure}[htb]
\centering
\includegraphics{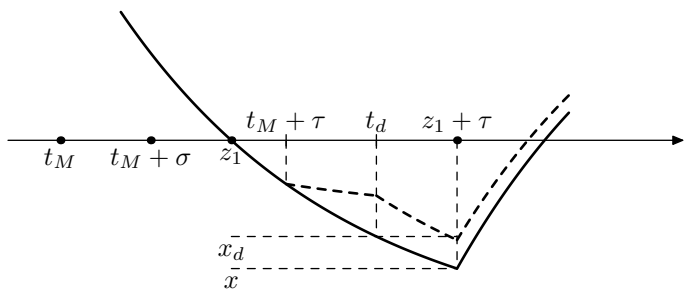}

\caption{A schematic representation of the ideas behind the arguments leading to a therapy plan.  The unperturbed limit cycle is the solid line while the desired limit cycle due to a perturbation is the dashed line.}
\label{fig:therapy}

\end{figure}

\medskip

Suppose there is a critical level $x_d$, larger than the minimum $\underline{x}$ of the
periodic oscillation in the patient without treatment. We want to find a therapy plan which consists of medication at certain times $t=t_M$ (which are to be determined) in such a way that the cell density in the patient never falls below the critical level.

\medskip

Medication at a time $t=t_M$ results in the begin of the production of more (precursors of) cells in the bone marrow, and this increased production lasts for a time interval of duration $\sigma>0$, from the time $t_M$ until the time $t_M+\sigma$. As in Section 4 the effect of medication at $t=t_M$ can be expressed by a 'temporal change' of the production function $f$, for example, by replacing $f$ by a sum $f_a=f+a$ with $a>0$ as long as $t_M\le t-\tau\le t_M+\sigma$. (Alternatively, one might replace $f$ by a multiple $f_a=af$ with $a>1$.) $a$ depends on the dose of the medication. Because of the delay $\tau$ due to the production process the number of cells in the bloodstream will begin to deviate from their number without treatment not earlier than $t\ge t_M+\tau$, where $t_M+\tau$ equals $\Delta$ in the cases studied in Section 5.

\medskip

We begin with  a simple situation and assume that for the time interval $[-\tau,0]$ in the past the cell density  in a patient is known, for example, by measurement, and
$$
0<x(t)\quad\text{for}\quad -\tau<t\le0.
$$
$t=0$ stands for the present time.
Then we use Eq.~\eqref{e:pl-sys-dimen1} in order to predict how the cell density would  evolve in the patient without treatment: We compute the solution $y(t)$, $t\ge0$, of  Eq.~\eqref{e:pl-sys-dimen1} with initial data $y(t)=x(t)$, $-\tau\le t\le0$. The solution $y(t)$ will be a translate of the unique slowly oscillating periodic solution $\tilde{x}$ of Eq.~\eqref{e:pl-sys-dimen1}. There is a first zero $z_1=z_1(y(0))>0$ of $y$, and there exists a unique time
$t_d$ between $z_1$ and $z_1+\tau$, at which $y$ reaches the critical level $x_d$,  $y(t_d)=x_d$, from above. (At $z_1+\tau$ $y$ attains its minimum value $\underline{x}<x_d$.)

\medskip

Having predicted the time $t_d$ we define $t_M=t_d-\tau-\sigma$ as the time of medication. If $t_M$ is positive (is in the future), then it is not too late for medication. After medication at $t=t_M$ the cell density  in the patient represented by $x(t)$ will equal $y(t)$ for $-\tau\le t\le t_M+\tau=t_d-\sigma$,
because of the delay in Eq.~\eqref{e:pl-sys-dimen1}. For $t_d-\sigma\le t\le t_d$ the release of cells into the circulation will be increased according to
$$
x'(t)=f(x(t-\tau))+a-x(t)=-\beta_U+a-x(t)
$$
while for $t\ge t_d$,  Eq.~\eqref{e:pl-sys-dimen1} holds once again.

\medskip

The question is whether for a range of parameters $a>0$ this can be done
 in such a way that for $t_d-\sigma\le t< z_1+\tau$ the solution $x(t)$ satisfies
$x_d\le x(t)<0$.

\medskip

If yes then $x(t)$ would increase after time $z_1+\tau$ until there is a zero  $z_2$, due to Eq.~\eqref{e:pl-sys-dimen1}, and would coincide on $[z_2,z_2+\tau]$ with the piece of the unique slowly oscillating periodic solution  $\tilde{x}$ of Eq.~\eqref{e:pl-sys-dimen1} before its maxima.

\medskip

Upon that the whole process can be repeated. It would result in a periodic therapy plan and a periodic solution $x(t)$ which never falls below the critical value $x_d$ and has a period shorter than the period of $\tilde{x}$. (This latter property comes from $x(z_1+\tau)\ge x_d>\underline{x}$.)

\medskip

Below we show that there exist parameters $\beta_L,\beta_U,\tau,\sigma$ and $x_d\in(\underline{x},0)$ and $a>0$ for which the program just described can be carried out. As initial data we consider continuous functions $\phi\colon[-\tau,0]\to\mathbb{R}$ with $0<\phi(t)$ for $-\tau<t\le0$.  Then for some $q$, $0<q<1$, $q\overline{x}<\phi(0)$. The solution $y$ of the initial value problem
$$
y'(t)=f(y(t-\tau))-y(t)\,\,\text{for}\,\,t>0,\,\,y_0=\phi
$$
has a first zero at
$$
z_1=z_1(\phi(0))=\ln\dfrac{\phi(0)+\beta_U}{\beta_U}
$$
and strictly decreases on $[z_1,z_1+\tau]$ to the value $\underline{x}$. For $\underline{x}<x_d<0$ we find a unique time $t=t_d=t_d(\phi)$ in $(z_1,z_1+\tau)$ with  $y(t_d)=x_d$, namely,
$$
t_d=z_1+\ln\dfrac{\beta_U}{x_d+\beta_U}.
$$

\medskip

Next we show that for parameters $-\beta_U<0<\beta_L,\tau>0,q\in(0,1)$ and $\sigma>0$ sufficiently small, and $x_d\in(\underline{x},0)$ sufficiently close to $\underline{x}$ we have
\begin{equation}
t_d-\tau-\sigma>0.
\label{tMpos}
\end{equation}
The inequality \eqref{tMpos} is equivalent to
$$
\tau+\sigma<\ln\dfrac{\phi(0)+\beta_U}{\beta_U}+\ln\dfrac{\beta_U}{x_d+\beta_U}=
\ln\dfrac{\phi(0)+\beta_U}{x_d+\beta_U},
$$
which follows from
$$
\tau+\sigma<\ln\dfrac{q\,\overline{x}+\beta_U}{x_d+\beta_U}.
$$
The preceding inequality can be achieved for $\sigma>0$ sufficiently small and $x_d>\underline{x}$ sufficiently close to $\underline{x}$ provided we have
\begin{equation}
\tau<\ln\dfrac{q\,\overline{x}+\beta_U}{\underline{x}+\beta_U}.
\label{(suftMpos)}
\end{equation}
We verify this: Using the equations for $\overline{x}$ and $\underline{x}$ from Corollary 3.2 we see that \eqref{(suftMpos)} is equivalent to
\begin{align*}
\tau & <   \ln\dfrac{q\,\beta_L(1-e^{-\tau})+\beta_U}{\beta_U-\beta_U(1-e^{-\tau})}\\
& =   \tau+\ln\dfrac{q\,\beta_L(1-e^{-\tau})+\beta_U}{\beta_U}\\
\end{align*}
or,
$$
\beta_U<q\,\beta_L(1-e^{-\tau})+\beta_U,
$$
which is equivalent to
\begin{equation}
0<q\,\beta_L(1-e^{-\tau}).
\end{equation}

\medskip

From now on assume that the  parameters $-\beta_U<0<\beta_L,\tau>0,\sigma>0$, and $q\in(0,1),x_d\in(\underline{x},0)$ satisfy \eqref{tMpos}. Assume in addition for simplicity that $\sigma$ is so small that we have
\begin{equation}
z_1<t_d-\sigma.
\end{equation}
Notice that $t_d-z_1=\ln\dfrac{\beta_U}{x_d+\beta_U}$ does not depend on $\phi$. We now define
$$
t_M=t_M(\phi)=t_d(\phi)-\tau-\sigma\,\,(>0)
$$
as the time of medication. For parameters $a>0$ we consider the
continuous function $x:[-\tau,z_1+\tau]\to\mathbb{R}$ which coincides with $y(t)$ for $-\tau\le t\le t_d-\sigma$ and satisfies
\begin{equation*}
x'(t) =
\left\{
  \begin{array}{ll}
    f(x(t-\tau))+a-x(t)=-\beta_U+a-x(t)\quad&\text{for}\,\,t_d-\sigma<t<t_d,\\
   f(x(t-\tau))-x(t)=-\beta_U-x(t)\quad &\text{for}\,\,t_d<t<z_1+\tau.
  \end{array}
\right.
\end{equation*}
It follows that
$$
x(t_d-\sigma)=y(t_d-\sigma)=-\beta_U+\beta_Ue^{-(t_d-\sigma-z_1)}=-\beta_U\left(1-e^{-\sigma}\dfrac{x_d+\beta_U}{\beta_U}\right).
$$
Similarly we get for $t\in[t_d-\sigma,t_d]$ that
\begin{align*}
x(t) & =  -\beta_U+a+\bigl(x(t_d-\sigma)-(-\beta_U+a))e^{-(t-(t_d-\sigma)\bigr)}\\
& =   -\beta_U+a+(y(t_d-\sigma)-(-\beta_U+a))e^{-(t-(t_d-\sigma))}\\
& =  y(t)+a(1-e^{-(t-(t_d-\sigma))})\ge y(t)\ge x_d
\end{align*}
which shows that $x$ is monotone and above $x_d$ on this interval. Using
$$
x(t_d-\sigma)=y(t_d-\sigma)<0
$$
and monotonicity we conclude that we have
$x(t)<0$ on $[t_d-\sigma,t_d]$ if and only if $x(t_d)<0$. Also,
$$
x(t_d)=y(t_d)+a\bigl(1-e^{-(t_d-(t_d-\sigma))}\bigr)=x_d+a(1-e^{-\sigma})
$$
which gives $x(t_d)<0$ if and only if
\begin{equation}
x_d+a(1-e^{-\sigma})<0.
\label{xdneg}
\end{equation}
We shall come back to this later, and turn to
\begin{align*}
x(z_1+\tau) & =  -\beta_U+\bigl(x(t_d)+\beta_U\bigr)e^{-(z_1+\tau-t_d)}\\
& =  -\beta_U+\bigl(y(t_d)+a(1-e^{-\sigma})+\beta_U\bigr)e^{-(z_1+\tau-t_d)}\\
& =  y(z_1+\tau)+a(1-e^{-\sigma})e^{-(z_1+\tau-t_d)}\\
& =  \underline{x}+a(1-e^{-\sigma})e^{-(z_1+\tau-t_d)}>\underline{x}.
\end{align*}
It follows that there is a unique $a=a_d>0$ so that
$$
x(z_1+\tau)=x_d>\underline{x},
$$
namely,
\begin{equation}
a_d  =  \dfrac{x_d-\underline{x}}{(1-e^{-\sigma})e^{-(z_1+\tau-t_d)}}
 =  \dfrac{(x_d-\underline{x})e^{\tau}(x_d+\beta_U)}{\beta_U(1-e^{-\sigma})}.
\end{equation}
We would like to have $x(t)<0$ on $(z_1,z_1+\tau]$. This follows from $x(z_1+\tau)=x_d<0$ in combination with monotonicity provided we have $x(t_d)<0$, which was characterized by \eqref{xdneg}. So we ask under which conditions $a=a_d$ satisfies \eqref{xdneg}, or equivalently,
$$
\dfrac{(x_d-\underline{x})e^{\tau}(x_d+\beta_U)}{\beta_U}=a_d(1-e^{-\sigma})<-x_d,
$$
which means
\begin{equation}
(x_d-\underline{x})e^{\tau}(x_d+\beta_U)<-\beta_Ux_d
.
\label{ad}
\end{equation}
Recall that $\underline{x}$ depends on $\tau$ and on $\beta_U$; given $\tau$ and $\beta_U$ the preceding inequality holds provided we consider $x_d\in(\underline{x},0)$ close enough to $\underline{x}$.

\medskip

Assume from now on that $x_d$ is chosen so that \eqref{ad} holds. If we follow the solution $x$ which started from $\phi$ further then we see from Eq.~\eqref{e:pl-sys-dimen1} and because of $x(t)<0$ on $(z_1,z_1+\tau]$ that $x$ begins to increase after $t=z_1+\tau$, has a first zero $z_2=z_2(\phi)>z_1+\tau$, and coincides on $[z_2,z_2+\tau]$ with a translate of the periodic solution $\tilde{x}$ of Eq.~\eqref{e:pl-sys-dimen1} which has a zero $t=z_2$ and then increases to the value $\overline{x}$ at $t=z_2+\tau$.
Notice that if we take this segment of $\tilde{x}$ as the initial value $\phi$ for the function $x$ then $x_{z_2+\tau}=\phi=x_0$, and iteration of the whole procedure yields a periodic solution $x$. The inequalities $x(t)<0$ for $z_1<t\le z_1+\tau$ and  $y(z_1+\tau)=\underline{x}<x_d=x(z_1+\tau)$ in combination with Eq.~\eqref{e:pl-sys-dimen1} for $z_1+\tau\le t$ imply that after the time $t=z_1+\tau$ the function $x$ reaches the zero level from below before $y$ does so, hence the period $z_2(\phi)+\tau$ of $x$ is {\it shorter} than the minimal period $\tilde{T}$ of $\tilde{x}$.

\medskip

\begin{remark}

(1) Crucial in the concept of the therapy plan is that the time $t_d$ at which the number $y(t)$ of blood cells in case of no medication would fall to the critical level $x_d$ is large enough for medication in the future (at some $t_M>0$) to become effective (during the time interval $[t_M+\tau,t_M+\tau+\sigma]$)
{\it before} the time $t=t_d$. Necessary for this is
$$
0<t_d-\tau, \,\,\text{or}\,\,\tau<t_d\,\,;
$$
the stronger condition \eqref{tMpos} which we used in the exposition above can be relaxed.

(2) A practically useful version of the therapy plan would require an extension to more realistic model equations, probably with continuous production functions, in order to get reliable predictions of $t_d=t_d(\phi)$ for a large set of initial conditions, as they may arise from monitoring the number of blood cells in a patient.

\end{remark}

\section{Discussion}\label{sec:disc}

The computation of the response of the periodic solution of \eqref{e:pl-sys-dimen1} to a perturbation of the form defined in Section \ref{sec:pulse} is complicated as we have seen in Section \ref{sec:pert}, and the response of the perturbed cycle length can be quite varied as shown in Section \ref{sec:prop-reset}.  However,  our calculations have shown that in a variety of situations that are dependent on the timing of the pulse, the pulse has had no effect on subsequent minima of the model solution.  This is in sharp contrast to what is noted in clinical situations where G-CSF is employed and in which both an amelioration as well as a  worsening of neutropenia is clearly documented in response to the G-CSF, and the nature of the response is dependent on when G-CSF is given as well as the dosage.  Although it seems to be a curious anomaly that a cytokine like G-CSF, which inhibits apoptosis, {\it should actually make neutropenia worse}, in this section  we will show  that in case the nonlinearity (production function) $f=f_{\ast}$ in Eq. (\ref{e:gen-dde})
is piecewise constant with {\it three} values (see Figure \ref{f:figf-4}), say,
$$
f_{\ast}(\xi)=\beta_L>0\,\,\text{for}\,\,\xi<0,\,\,f_{\ast}(\xi)=-\beta_U<0\,\,\text{for}\,\,0\le\xi<\xi_{\ast}
$$
and
$$
f_{\ast}(\xi)=-\beta_{\ast}<-\beta_U\,\,\text{for}\,\,\xi_{\ast}\le\xi,
$$
then a pulse as in case $\mathbf{RPRP}$
can result in a subsequent minimum of the solution which is {\it lower} than the minimum of a periodic solution of the equation without a pulse. This will happen if the pulse pushes the state variable to high values where the negative feedback is so strong that after the delay time it drives the state down to very low values.

\medskip

\begin{figure}[htb]
\centering

\includegraphics{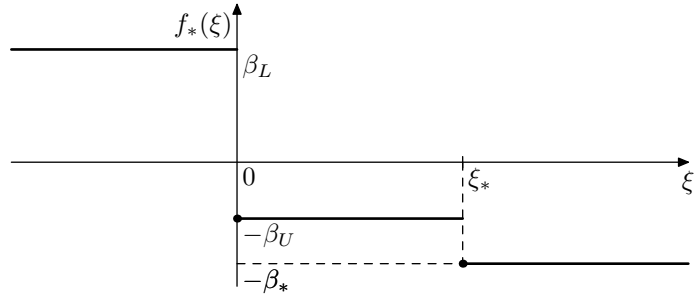}
\caption{The graph of the function $f_{\ast}$}
\label{f:figf-4}
\end{figure}

\medskip

In order to see that this actually happens for a suitable range of parameters, assume (in part for simplicity)
that $\xi_{\ast}=\overline{x}$, and define solutions of the equation
\begin{equation}
x'(t)=-x(t)+f_{\ast}(x(t-\tau)),
\label{e:6.32}
\end{equation}
i.e., of Eq. (\ref{e:gen-dde}) with $f=f_{\ast}$, as in Section \ref{sec:global}.
Then our former periodic solution $\tilde{x}$ of Eq. (\ref{e:pl-sys-dimen1})  will also be a solution of Eq. (\ref{e:6.32}).
Consider a  pulse which begins at $\tilde{z}_1$ and ends at $\tilde{z}_1+\tau$, that is, consider the function $x_{\ast}:\mathbb{R}\to\mathbb{R}$
which coincides with $\tilde{x}$ on $(-\infty,\tilde{z}_1]$, is given by
$$
x'(t)=-x(t)+\beta_L+a\,\,\text{for}\,\,\tilde{z}_1\le t\le\tilde{z}_1+\tau,
$$
and by Eq. (\ref{e:6.32}) for $t\ge\tilde{z}_1+\tau$.
We have
$$
x_{\ast}(\tilde{z}_1+\tau)=(\beta_L+a)(1-e^{-\tau})>\beta_L(1-e^{-\tau})=\tilde{x}(\tilde{z}_1+\tau)=\overline{x}=\xi_{\ast}.
$$
Incidentally, notice that $x_{\ast}(\tilde{z}_1+\tau)\to\beta_L+a$ as $\tau\to\infty$.

The first time $t_{\ast}>\tilde{z}_1$ at which $x_{\ast}$ crosses the level $\xi_{\ast}$ from below is given by
$$
\xi_{\ast}=(\beta_L+a)(1-e^{-(t_{\ast}-\tilde{z}_1)}),
$$
or equivalently,
$$
\beta_L(1-e^{-\tau})=(\beta_L+a)(1-e^{-(t_{\ast}-\tilde{z}_1)}),
$$
hence
$$
e^{\tilde{z}_1-t_{\ast}}=\dfrac{a+\beta_Le^{-\tau}}{a+\beta_L}.
$$
Notice here that
$$
e^{\tilde{z}_1-t_{\ast}}\to\dfrac{a}{a+\beta_L}\,\,\text{as}\,\,\tau\to\infty.
$$
Next,
\begin{align*}
x_{\ast}(t_{\ast}+\tau) & = -\beta_U+\bigl(x_{\ast}(\tilde{z}_1+\tau)+\beta_U\bigr)e^{-(t_{\ast}+\tau-(\tilde{z}_1+\tau))}\\
& =  -\beta_U+\bigl(x_{\ast}(\tilde{z}_1+\tau)+\beta_U\bigr)e^{\tilde{z}_1-t_{\ast}}.
\end{align*}
We observe that $x_{\ast}(t_{\ast}+\tau)$ has a limit as $\tau\to\infty$.
It follows that
\begin{align*}
x_{\ast}(\tilde{z}_1+\tau+\tau) & = -\beta_{\ast}+\bigl(x_{\ast}(t_{\ast}+\tau)+\beta_{\ast}\bigr)e^{-(\tilde{z}_1+2\tau-(t_{\ast}+\tau))}\\
 & =  -\beta_{\ast}+\bigl(x_{\ast}(t_{\ast}+\tau)+\beta_{\ast}\bigr)e^{t_{\ast}-\tilde{z}_1}e^{-\tau}
\end{align*}
converges to $-\beta_{\ast}<-\beta_U<\underline{x}=\min\,\tilde{x}(\mathbb{R})$ as $\tau\to\infty$. So, given
$\xi_{\ast}>0$ and $\beta_L>0>-\beta_U>-\beta_{\ast}$ and $a>0$ there exists $\tau_0>0$ so that for all $\tau\ge\tau_0$ the solution $x_{\ast}$ assumes values strictly less than $\underline{x}=\min\,\tilde{x}(\mathbb{R})$.

Our investigations in this paper have been confined to an examination of the response of the limit cycle solution of \eqref{e:pl-sys-dimen1} to a {\it single} perturbation.  However, in many situations of interest biologically (and certainly for the clinical questions that motivated this study) one is interested in the limiting behaviour of the limit cycle in response to periodic perturbations, c.f  \citet{winfree80,guevara1982,glass-winfree-1984,trine2004,bodnar2013b} for representative examples.  However, considerations of the response to the system we have studied to periodic perturbation is quite beyond the scope of this study as it would entail the development of completely different techniques than the ones that have proved so successful in the study of the response to single perturbations.  We are of the opinion that deriving the phase response curve in the face of periodically delivered pulses will only be possible, in general, for certain limiting cases of the pulse parameters, namely $\sigma \simeq 0$ and, possibly, small values of the amplitude $a$.  It is  possible that techniques such as those employed in \citet{bard2012} may be useful in this regard.

\section*{Acknowledgments}

MCM would like to thank the Universities of Bremen and Giessen and the Fields Institute, Toronto, for their hospitality during the time that some of this work was carried out. H-OW thanks McGill University for hosting his visit in September and October, 2014. This research was supported by the NSERC (Canada) and the Polish NCN grant no 2014/13/B/ST1/00224.  We are grateful to Prof. Bard Ermentrout (Pittsburg) for preliminary discussions, and to Dr. Daniel C\^{a}mara de Souza for his careful reading of the manuscript and pointing out some errors.

\appendix
\section{Proofs of the results from Sections \ref{sec:global} and \ref{sec:pulse}}

\begin{proof}[Proof of Proposition~\ref{prop:prop3.1}]

\begin{enumerate}
\item We begin with continuity of the {\it time-$\tau$-map}
$$
S(\tau,\cdot):Z\ni\phi\mapsto x^{\phi}_{\tau}\in Z.
$$
Observe that for $0\le t\le\tau$,
$$
x^{\phi}(t)=e^{-t}\phi(0)+\int_0^te^{-(t-s)}f(\phi(s-\tau))ds.
$$
For $\psi$ and $\phi$ in $Z$ and $0\le t\le\tau$ we have
$$
|x^{\psi}(t)-x^{\phi}(t)|\le|\phi(0)-\psi(0)|+\int_{-\tau}^0|f(\psi(s))-f(\phi(s))|ds
$$
where the integrand is nonzero only on the set
$$
N(\psi,\phi)=\{t\in[0,\tau]: \mathrm{sign}(\psi(t))\neq \mathrm{sign}(\phi(t))\}.
$$
It follows that
$$
|x^{\psi}_{\tau}-x^{\phi}_{\tau}|_C\le|\psi(0)-\phi(0)|+\beta\lambda(N(\psi,\phi)),
$$
with the Lebesgue measure $\lambda$ and a positive constant $\beta$. It is easy to see that
$$
\lim_{Z\ni\psi\to\phi\in Z}\lambda(N(\psi,\phi))=0.
$$
(Proof of this in case $\phi\in Z$ has zeros $z_1<z_2<\ldots<z_J$. Let $\epsilon>0$ be given. The complement of the set
$$
\bigcup_{j=1}^J\left(z_j-\dfrac{\epsilon}{2J},z_j+\dfrac{\epsilon}{2J}\right)
$$
in $[-\tau,0]$ is the finite union of compact intervals on each of which $\phi$ is either strictly positive, or strictly negative. There exists $\delta>0$ so that for every $\psi\in Z$ with $|\psi-\phi|<\delta$ the signs of $\psi(t)$ and $\phi(t)$ coincide on each of the compact intervals. This yields
$$
\lambda(N(\psi,\phi))\le\sum_{j=1}^J2\dfrac{\epsilon}{2J}=\epsilon.)
$$
Then it follows easily that
$$
\lim_{Z\ni\psi\to\phi\in Z}(S(\tau,\psi)-S(\tau,\phi))=\lim_{Z\ni\psi\to\phi\in Z}(x^{\psi}_{\tau}-x^{\phi}_{\tau})=0.
$$

\item Iterating we find that for every integer $n>0$ the time-$n\tau$-map $S(n\tau,\cdot)$ is continuous. Having this we obtain continuous dependence on initial data in the sense that for every $t\ge0$ and $\phi\in Z$,
$$
\lim_{Z\ni\psi\to\phi\in Z}\max_{-\tau\le s\le t}|x^{\psi}(s)-x^{\phi}(s)|=0.
$$

Finally, the continuity of $S$ at $(t,\phi)\in[0,\infty)\times Z$ follows by means of the estimate
\begin{align*}
|S(s,\psi)-S(t,\phi)| & \le |S(s,\psi)-S(s,\phi)|+|S(s,\phi)-S(t,\phi)|\\
& \le  \max_{-\tau\le v\le t+1}|x^{\psi}(v)-x^{\phi}(v)|+\max_{-\tau\le w\le0}|x^{\phi}(s+w)-x^{\phi}(t+w)|
\end{align*}
for $0\le s\le t+1$ and $\psi\in Z$ from continuous dependence on initial data as before  in combination with the uniform continuity of $x^{\phi}$ on $[-\tau,t+1]$.
\end{enumerate}
\end{proof}

\begin{proof}[Proof of Proposition \ref{prop:prop4.2}]
By definition the value of the cycle length map at $\Delta$  is $T(\Delta)=z-\tilde{z}_J=z-z_{\Delta,J}$ where $z$ is the smallest zero of $x^{(\Delta)}$ in $(\tilde{z}_J,\infty)$ such that $x^{(\Delta)}(z+t)=\tilde{x}(\tilde{z}_J+t)$ for all $t\ge0$. We have
$$
\text{sign}(\tilde{x}(\tilde{z}_J+t))=\text{sign}(x^{(\Delta)}(z_{\Delta,J}+t))=-\text{sign}(x^{(\Delta)}(z_{\Delta,J+1}+t))
$$
for $0<t\le\tau$ since $x^{(\Delta)}$ changes sign at each zero. We infer that $z>z_{\Delta,J+1}$. Notice that the definition of $J=j_{\Delta}$ implies $\Delta<z_{\Delta,J+1}$. Hence the next zero $z_{\Delta,J+2}$ satisfies $z_{\Delta,J+2}>z_{\Delta,J+1}+\tau>\Delta+\tau\ge\Delta+\sigma$. Therefore on $(z_{\Delta,J+2},\infty)$ the function $x^{(\Delta)}$  is given by Eq. (\ref{e:pl-sys-dimen1}), and satisfies
$$
\text{sign}(x^{(\Delta)}(z_{\Delta,J+2}+t))=-\text{sign}(x^{(\Delta)}(z_{\Delta,J+1}+t))=\text{sign}(\tilde{x}(\tilde{z}_J+t))
$$
for $0<t\le\tau$.
This yields $x^{(\Delta)}(z_{\Delta,J+2}+t)=\tilde{x}(\tilde{z}_J+t)$ for all $t\ge0$.
\end{proof}

\begin{proposition}
For $\Delta=\tilde{z}_J$, $J=1$ or $J=2$, we have $T(\Delta)=z_{\Delta,J+1}-z_{\Delta,J-1}$.
\label{prop:4.3}
\end{proposition}

\begin{proof}
From $\Delta=\tilde{z}_J$ we obtain $\Delta+\sigma\le\Delta+\tau<z_{\Delta,J+1}$. This implies that for $t\ge z_{\Delta,J+1}$ the function $x^{(\Delta)}$ satisfies Eq. (\ref{e:pl-sys-dimen1}). Using this and the fact that $x^{(\Delta)}$ and $\tilde{x}$
change sign at $\tilde{z}_{J-1}=z_{\Delta,J-1}$ and at $\tilde{z}_J=z_{\Delta,J}$ respectively  we infer that for all $t\ge z_{\Delta,J+1}$ we have $x^{(\Delta)}(z_{\Delta,J+1}+t)=\tilde{x}(\tilde{z}_{J-1}+t)$. It follows that
$$
z_{\Delta,J+2}=z_{\Delta,J+1}+(\tilde{z}_J-\tilde{z}_{J-1}).
$$
Combining this with Proposition~\ref{prop:prop4.2} we find
$$
T(\Delta)=z_{\Delta,J+2}-\tilde{z}_J=z_{\Delta,J+1}-\tilde{z}_{J-1}=z_{\Delta,J+1}-z_{\Delta,J-1}.
$$
\end{proof}

\begin{proof}[Proof of Corollary \ref{cor:cor4.2}]
Let $\Delta_0\in[0,\tilde{T})$ be given and set $J=j(\Delta_0)$. Then
$\Delta_0<z_{\Delta_0,J+1}$. Corollary \ref{cor:contin} yields a neighbourhood $N$ of $\Delta_0$ in $[0,\tilde{T})$ such that for all $\Delta\in N$ we have $\Delta<z_{\Delta,J+1}$.

\begin{enumerate}
\item The case $\tilde{z}_J<\Delta_0$. Then by Corollary \ref{cor:contin}, $\tilde{z}_J<\Delta<z_{\Delta,J+1}$ for all $\Delta$ in a neighbourhood $V\subset N$ of $\Delta_0$ in $[0,\tilde{T})$. For $\Delta\in V$ we get $j(\Delta)=J$, hence $T(\Delta)=z_{\Delta,J+2}-\tilde{z}_J$, and Corollary \ref{cor:contin} yields continuity at $\Delta_0$.

\item The case $\tilde{z}_J=\Delta_0$. There is a neighbourhood $U\subset N$ of $\Delta_0$ in $[0,\tilde{T})$ with
$\tilde{z}_{J-1}<\Delta$ for all $\Delta\in U$. For all $\Delta\in U$ with $\Delta<\tilde{z}_J$  this yields $j(\Delta)=J-1$ and $T(\Delta)=z_{\Delta,J-1+2}-\tilde{z}_{J-1}$ . At $\Delta=\Delta_0$
we have
\begin{align*}
T(\Delta)(\Delta_0) & =  z_{\Delta_0,J+2}-\tilde{z}_J\\
& =  z_{\Delta_0,J+1}-z_{\Delta_0,J-1}\quad\text{(see Proposition \ref{prop:4.3})}\\
& =  z_{\Delta_0,J+1}-\tilde{z}_{J-1}.
\end{align*}
The continuity of the map $\Delta\mapsto z_{\Delta,J+1}$ due to Corollary \ref{cor:contin} now shows that the restriction of the cycle length map to the set $[0,\tilde{z}_J]\cap U$ is continuous.
For $\tilde{z}_J\le\Delta\in U\subset N$  we have $\tilde{z}_J\le\Delta<z_{\Delta,J+1}$, hence $j(\Delta)=J$, and thereby
$T(\Delta)=z_{\Delta,J+2}-\tilde{z}_J$. The continuity of the map $\Delta\mapsto z_{\Delta,J+2}$ due to Corollary \ref{cor:contin} shows that the restriction of the cycle length map to the set $U\cap[\tilde{z}_J,\tilde{T})$ is continuous. As both restrictions coincide at $\tilde{z}_J=\Delta_0$ we obtain continuity of the cycle length map at $\Delta_0$.
\end{enumerate}
\end{proof}

\begin{proof}[Proof of Proposition~\ref{pr:prop4.4}]
\begin{enumerate}
\item Let $\Delta_0\in[0,\tilde{T})$ be given. Set $J=j(\Delta)$. Then $\tilde{z}_J\le\Delta_0<z_{\Delta_0,J+1}$. Using
Corollary \ref{cor:contin} we find a neighbourhood $N$ of $\Delta_0$ in $[0,\tilde{T})$ such that for every $\Delta\in N$ we have
$$
\tilde{z}_{J-1}<\Delta<z_{\Delta,J+1}.
$$
In the following we show continuity of the map $[0,\tilde{T})\ni\Delta\mapsto\overline{x}_{\Delta}\in\mathbb{R}$. The proof for the other map is analogous.

\item  For $\Delta\in N\cap[\tilde{z}_J,\infty)$ we have $J=j(\Delta)$, hence
$$
\overline{x}_{\Delta}=\max_{\tilde{z}_J\le t\le z_{\Delta,J+2}}x^{(\Delta)}(t).
$$
Using this, the uniform continuity of the map
$$
[0,\tilde{T})\times[0,\infty)\ni(\Delta,t)\mapsto x^{(\Delta)}(t)\in\mathbb{R},
$$
on compact sets, and the continuity of the map $\Delta\mapsto z_{\Delta,J+2}$ (see Corollary \ref{cor:contin}) one can easily show that the map
$$
N\cap[\tilde{z}_J,\infty)\ni\Delta\mapsto\overline{x}_{\Delta}\in\mathbb{R}
$$
is continuous.

Similarly we have for $\Delta\in N\cap(-\infty,\tilde{z}_J)$ that $J-1=j(\Delta)$, hence
$$
\overline{x}_{\Delta}=\max_{\tilde{z}_{J-1}\le t\le z_{\Delta,J+1}}x^{(\Delta)}(t).
$$
As before one can then easily show that the map
$$
N\cap(-\infty,\tilde{z}_J)\ni\Delta\mapsto\overline{x}_{\Delta}\in\mathbb{R}
$$
is continuous.

\item It remains to prove that in case $\Delta_0=\tilde{z}_J$ (where  $N\cap(-\infty,\tilde{z}_J)\neq\emptyset$) we have $\overline{x}_{\Delta}\to\overline{x}_{\Delta_0}$ as $\Delta\nearrow\Delta_0$.
\begin{enumerate}
\item The case $\Delta_0=\tilde{z}_J$ and $\tilde{x}'(\tilde{z}_J)<0$. Using the fact that $x^{(\Delta_0)}$ changes sign at each zero we obtain $x^{(\Delta_0)}(t)\le0$ on $[\tilde{z}_j,z_{\Delta_0,J+1}]=[\Delta_0,z_{\Delta_0,J+1}]$, and
\begin{align*}
\overline{x}_{\Delta_0} & =  \max_{\tilde{z}_J\le t\le z_{\Delta_0,J+2}}x^{(\Delta_0)}(t)\\
& =  \max_{z_{\Delta_0,J+1}\le t\le z_{\Delta_0,J+2}}x^{(\Delta_0)}(t)\\
& =  \overline{x}
\end{align*}
where the last equation holds because $\Delta_0=\tilde{z}_J$  implies $\Delta_0+\sigma<\Delta_0+\tau<z_{\Delta_0,J+1}$ and thereby $x^{(\Delta_0)}(z_{\Delta_0,J+1}+t)=\tilde{x}(\tilde{z}_{J-1}+t)$ for all $t\ge0$.

Using continuity as in Part 2 above we find a neighbourhood $U\subset N$ of $\Delta_0$ in $[0,\tilde{T})$ such that for each $\Delta\in U$ we have
$$
x^{(\Delta)}(t)<\dfrac{1}{2}\overline{x}\quad\text{on}\quad[\Delta,z_{\Delta,J+1}]
$$
and $\tilde{z}_{J-1}+\tau<\Delta$. For $\Delta\in U$ with $\Delta<\tilde{z}_J$, we have $J-1=j(\Delta)$, and the preceding inequality yields
$$
x^{(\Delta)}(\tilde{z}_{J-1}+\tau)=\tilde{x}(\tilde{z}_{J-1}+\tau)=\overline{x}\quad\left(>\dfrac{1}{2}\overline{x}\right).
$$
It follows that
\begin{align*}
\overline{x}_{\Delta} & = \max_{\tilde{z}_{J-1}\le t\le z_{\Delta,J+1}}x^{(\Delta)}(t)\\
& = \max_{\tilde{z}_{J-1}\le t\le\Delta}x^{(\Delta)}(t)\\
& = \overline{x},
\end{align*}
so the map $U\cap(-\infty,\tilde{z}_J)\ni\Delta\to\overline{x}_{\Delta}\in\mathbb{R}$ is constant with value
$\overline{x}=\overline{x}_{\Delta_0}$.

\item The case $\Delta_0=\tilde{z}_J$ and $\tilde{x}'(\tilde{z}_J)>0$. Then $x^{(\Delta_0)}$ is negative on $(\tilde{z}_{J-1},\tilde{z}_J)$, positive on $(\tilde{z}_J,z_{\Delta_0,J+1})$ and negative on $(z_{\Delta_0,J+1},z_{\Delta_0,J+2})$, and
$$
\overline{x}_{\Delta_0}=\max_{\tilde{z}_J\le t\le z_{\Delta_0,J+2}}x^{(\Delta)}(t)=\max_{\tilde{z}_J\le t\le z_{\Delta_0,J+1}}x^{(\Delta)}(t)>0.
$$
Choose $t_0\in(\tilde{z}_J,z_{\Delta_0,J+1})$ with
$$
x^{(\Delta)}(t_0)=\overline{x}_{\Delta_0}>0.
$$
By continuity there exists a neighbourhood $V\subset N$ of $\Delta_0$ in $[0,\tilde{T})$ such that for every $\Delta\in V$ we have
\begin{align*}
x^{(\Delta)}(t) & < \dfrac{1}{2}\overline{x}_{\Delta_0}\quad\text{on}\quad[\tilde{z}_{J-1},z_{\Delta,J}],\\
z_{\Delta,J} & < t_0<z_{\Delta,J+1},\\
\dfrac{1}{2}\overline{x}_{\Delta_0} & <  x^{(\Delta)}(t_0).
\end{align*}
For $\Delta\in V$ with $\Delta<\tilde{z}_J$ we have $j(\Delta)=J-1$, and we conclude that
\begin{align*}
\overline{x}_{\Delta} & =  \max_{\tilde{z}_{J-1}\le t\le z_{\Delta,J+1}}x^{(\Delta)}(t)\\
& =  \max_{z_{\Delta,J}\le t\le z_{\Delta,J+1}}x^{(\Delta)}(t)
\end{align*}
where by continuity the last term converges to
$$
\max_{z_{\Delta_0,J}\le t\le z_{\Delta_0,J+1}}x^{(\Delta_0)}(t)=\overline{x}_{\Delta_0}
$$
as $V\ni\Delta\nearrow\Delta_0$.
\end{enumerate}
\end{enumerate}
\end{proof}

\section{Proofs of the results from Section~\ref{sec:pert}}

\begin{proof}[Proof of Proposition \ref{c:prop1}]
First we  show that  $x^{(\Delta)}$ has a first positive zero $z_{\Delta,1}<\tilde{z}_1$.
We have $x^{(\Delta)}(t)<0$ for  $t \in (-\tau,\Delta+\sigma)$. For $t \in [\Delta+\sigma,\infty)$, $x^{(\Delta)}(t)$ is given by \eqref{e:after} as long as $x^{(\Delta)}(t-\tau)<0$.  Compute
$$
z_{\Delta,1} = \Delta+\sigma+\ln\dfrac{\beta_L-x^{(\Delta)}(\Delta+\sigma)}{\beta_L} \ge\Delta+\sigma
$$
from the condition $x^{(\Delta)}(z_{\Delta,1})=0$. Similarly since $\tilde{x}(\tilde{z}_1) = 0$ we obtain
$$
\tilde{z}_1=\Delta+\sigma+\ln\dfrac{\beta_L-\tilde{x}(\Delta+\sigma)}{\beta_L}.
$$
Since $\tilde{x}(\Delta+\sigma)<x^{(\Delta)}(\Delta+\sigma)$ we have $z_{\Delta,1}<\tilde{z}_1$.

The largest zero of $\tilde{x}$ on $(-\infty,\Delta]$ is $\tilde{z}_0=-\tau$. Hence the minimal value of $x^{(\Delta)}$ on $[-\tau,z_{\Delta,1}]$ is equal to $\underline{x}$.
On the interval $[z_{\Delta,1},z_{\Delta,1}+\tau]$, $x^{(\Delta)}(t)$ is given by \eqref{e:after}. This yields $x^{(\Delta)}(t+z_{\Delta,1})=\tilde{x}(t+\tilde{z}_1)$ for all $t\ge 0$.
It follows that $\overline{x}_{\Delta}=\overline{x}$ and $\underline{x}_{\Delta}=\underline{x}$ and
\begin{align*}
T(\Delta)  & =  \bigl(z_{\Delta,1}+(\tilde{z}_2-\tilde{z}_1)\bigr)-\tilde{z}_{0}=z_{\Delta,1}-\tilde{z}_1+\tilde{T}\\
& =  \tilde{T}+\ln\dfrac{\beta_L-x^{(\Delta)}(\Delta+\sigma)}{\beta_L}-\ln\dfrac{\beta_L-\tilde{x}(\Delta+\sigma)}{\beta_L}\\
& =  \tilde{T}+\ln\dfrac{\beta_L-x^{(\Delta)}(\Delta+\sigma)}{\beta_L-\tilde{x}(\Delta+\sigma)}.
\end{align*}
The formula \eqref{e:risingxDeltasigma} yields
$$
x^{(\Delta)}(\Delta + \sigma)-\beta_L=  -\beta_L  e^{\tilde{z}_1-\Delta-\sigma}+a(1-e^{-\sigma})
$$
and with
$$
\tilde{x}(\Delta+\sigma)-\beta_L=(\underline{x}-\beta_L)e^{-(\Delta+\sigma)}=-\beta_Le^{\tilde{z}_1-(\Delta+\sigma)}
$$
we find
\begin{align*}
T(\Delta) & =  \tilde{T}+\ln\dfrac{ \beta_L  e^{\tilde{z}_1-\Delta-\sigma}-a(1-e^{-\sigma})}{\beta_Le^{\tilde{z}_1-(\Delta+\sigma)}}\\
& =  \tilde{T}+\ln\left(1-\dfrac{a(1-e^{-\sigma})}{\beta_L}e^{\Delta+\sigma-\tilde{z}_1}\right).
\end{align*}
Thus, $T(\Delta) < \tilde{T}$, and the restriction of $T(\Delta)$  to $I_{RNRN}$ is strictly decreasing.
\end{proof}

From \eqref{e:after} it follows that $x^{(\Delta)}$ is strictly decreasing right after $t=\Delta+\sigma$ if and only if $x^{(\Delta)}(\Delta+\sigma)-\beta_L>0$. We have,   by \eqref{e:risingxDeltasigma},
\begin{equation*}
x^{(\Delta)}(\Delta+\sigma)-\beta_L=-\beta_Le^{\tilde{z}_1-\Delta-\sigma}+a(1-e^{-\sigma}).
\end{equation*}
Hence, $x^{(\Delta)}(\Delta+\sigma)-\beta_L>0$ if and only if
\[
e^{\Delta}>\dfrac{\beta_Le^{\tilde{z}_1-\sigma}}{a(1-e^{-\sigma})}.
\]
Let us define
\begin{equation}\label{d:hatdelta1}
\hat{\delta}_1=\ln \dfrac{\beta_Le^{\tilde{z}_1-\sigma}}{a(1-e^{-\sigma})}.
\end{equation}
We have
\[
\hat{\delta}_1=\tilde{z}_1-\sigma+\ln \dfrac{\beta_L}{a(1-e^{-\sigma})}>\delta_1.
\]
Thus
we divide the case \textbf{RNRP} when $x^{(\Delta)}(\Delta+\sigma)\ge 0$  into the two subcases
$0\le x^{(\Delta)}(\Delta+\sigma)\le \beta_L$ and $x^{(\Delta)}(\Delta+\sigma)>\beta_L$ and we consider these two subcases separately as  $\mathbf{RNRP1}$ and $\mathbf{RNRP2}$, see Figures~\ref{fig:bl-a2} and~\ref{fig:bl-a2b}, respectively. The $\Delta$-intervals are of the form
\[
I_{RNRP1}=[\max\{0,\delta_1\},\tilde{z}_1)\cap[\delta_1,\hat{\delta}_1]\quad \text{and}\quad I_{RNRP2}=[\max\{0,\delta_1\},\tilde{z}_1)\cap(\hat{\delta}_1,\infty).
\]

\begin{description}
\item[Case $\mathbf{RNRP1}$] The pulse parameters $(a,\Delta,\sigma)$ are such that $\Delta<\tilde{z}_1$ and $x^{(\Delta)}(\Delta+\sigma)\in [0,\beta_L]$. Then $x^{(\Delta)}$  is increasing right after the pulse.
\item[Case $\mathbf{RNRP2}$] The pulse parameters $(a,\Delta,\sigma)$ are such that $\Delta<\tilde{z}_1$ and $x^{(\Delta)}(\Delta+\sigma)>\beta_L$. Then $x^{(\Delta)}$ is strictly decreasing right after the pulse.
\end{description}

\begin{proof}[Proof of Proposition~\ref{c:prop2}]
The equation
$$
0=x^{(\Delta)}(z_{\Delta,1})=\beta_L+a+\bigl(x^{(\Delta)}(\Delta)-(\beta_L+a)\bigr)e^{-(z_{\Delta,1}-\Delta)}
$$
together with \eqref{e:risingxDelta} yields
\begin{align}\nonumber
(\beta_L+a)e^{z_{\Delta,1}}&=\bigl(\beta_L+a-x^{(\Delta)}(\Delta)\bigr)e^{\Delta}\\
&=\beta_Le^{\tilde{z}_1}+ae^{\Delta}.\label{d:rez1p}
\end{align}
Use $\Delta<z_{\Delta,1}$ and $a>0,\beta_L>0$ to obtain
\[
e^{z_{\Delta,1}}-e^{\tilde{z}_1}=\dfrac{a}{\beta_L}(e^{\Delta}-e^{z_{\Delta,1}})<0
\]
and to conclude that $z_{\Delta,1}<\tilde{z}_1$.

We have
\begin{equation*}
x^{(\Delta)}(z_{\Delta,1}+\tau)=\beta_L+(x^{(\Delta)}(\Delta+\sigma)-\beta_L)e^{-\left(z_{\Delta,1}+\tau-
(\Delta+\sigma)\right)},
\end{equation*}
which by \eqref{e:risingxDeltasigma} can be rewritten as
\begin{align*}
x^{(\Delta)}(z_{\Delta,1}+\tau)&=\beta_L+\bigl(-\beta_Le^{\tilde{z}_1-\Delta-\sigma}+a(1-e^{-\sigma})\bigr)
e^{-\left(z_{\Delta,1}+\tau-
(\Delta+\sigma)\right)}\\
&=\beta_L-(\beta_Le^{\tilde{z}_1}+ae^{\Delta})e^{-(z_{\Delta,1}+\tau)}+ae^{-\left(z_{\Delta,1}+\tau-
(\Delta+\sigma)\right)}.
\end{align*}
From  \eqref{d:rez1p}  it follows that
\begin{equation}\label{e:xtmax}
x^{(\Delta)}(z_{\Delta,1}+\tau)=\beta_L-(\beta_L+a)e^{-\tau}+ae^{-\tau}e^{\sigma+\Delta-z_{\Delta,1}}
\end{equation}
and that the function $[\max\{0,\delta_1\},\tilde{z}_1)\ni\Delta\mapsto e^{z_{\Delta,1}-\Delta}\in\mathbb{R}$ is strictly decreasing, which shows that $x^{(\Delta)}(z_{\Delta,1}+\tau)$ is strictly increasing with respect to $\Delta\in[\max\{0,\delta_1\},\tilde{z}_1) $.  Since $\Delta+\sigma\ge z_{\Delta,1}$ and $ae^{-\tau}>0$, we get
\[
x^{(\Delta)}(z_{\Delta,1}+\tau)\ge \beta_L-\beta_Le^{-\tau}=\overline{x}.
\]
Also the function $[\max\{0,\delta_1\},\tilde{z}_1)\ni\Delta\mapsto x^{(\Delta)}(\Delta+\sigma)\in\mathbb{R}$ is increasing.

In subcase $\mathbf{RNRP1}$ the function $x^{(\Delta)}$ is increasing on $[\Delta+\sigma,z_{\Delta,1}+\tau]$ while in subcase $\mathbf{RNRP2}$ it is decreasing on that interval. It follows that in subcase $\mathbf{RNRP1}$,
$$
\max_{z_{\Delta,1}\le t\le z_{\Delta,1}+\tau}x^{(\Delta)}(t)=x^{(\Delta)}( z_{\Delta,1}+\tau)\ge\overline{x}
$$
while in subcase $\mathbf{RNRP2}$,
$$
\max_{z_{\Delta,1}\le t\le z_{\Delta,1}+\tau}x^{(\Delta)}(t)=x^{(\Delta)}( \Delta+\sigma)\ge x^{(\Delta)}(z_{\Delta,1}+\tau)\ge\overline{x}.
$$
In both subcases $\max_{z_{\Delta,1}\le t\le z_{\Delta,1}+\tau}x^{(\Delta)}(t)$ is increasing with respect to $\Delta$.

Also in both subcases we have $0<x^{(\Delta)}(t)$ for $z_{\Delta,1}<t\le z_{\Delta,1}+\tau$. It follows that after $t=z_{\Delta,1}+\tau\ge\Delta+\sigma$ the function $x^{(\Delta)}$ is given by
$$
x^{(\Delta)}(t)=-\beta_U+\bigl(x^{(\Delta)}(z_{\Delta,1}+\tau)+\beta_U\bigr)e^{-\left(t-(z_{\Delta,1}+\tau)\right)}
$$
as long as $x^{(\Delta)}(t-\tau)>0$. We obtain a first zero $z_{\Delta,2}$ of $x^{(\Delta)}$ in $(z_{\Delta,1}+\tau,\infty)$, and
for all $t\ge0$, $x^{(\Delta)}(z_{\Delta,2}+t)=\tilde{x}(\tilde{z}_0+t)$  (recall $\tilde{z}_0=-\tau$). Then
$$
T(\Delta)=z_{\Delta,2}-\tilde{z}_0=z_{\Delta,2}+\tau.
$$
Moreover,
\begin{align*}
\overline{x}_{\Delta} & =  \max_{-\tau\le t\le z_{\Delta,2}}x^{(\Delta)}(t)=\max_{z_{\Delta,1}\le t\le z_{\Delta,1}+\tau}x^{(\Delta)}(t)\\
& =  \max\{x^{(\Delta)}(\Delta+\sigma),x^{(\Delta)}(z_{\Delta,1}+\tau)\}\\
& \ge  \overline{x}
\end{align*}
is increasing with respect to $\Delta$ in both subcases.

Now the equation
$$
0=  x^{(\Delta)}(z_{\Delta,2}) =  -\beta_U+\bigl(x^{(\Delta)}(z_{\Delta,1}+\tau)+\beta_U\bigr)e^{z_{\Delta,1}+\tau-z_{\Delta,2}}
$$
yields
$$
\beta_U e^{z_{\Delta,2}}=\bigl(x^{(\Delta)}(z_{\Delta,1}+\tau)+\beta_U\bigr)e^{z_{\Delta,1}+\tau}.
$$
Also, from \eqref{d:rez1p} and \eqref{e:xtmax} we obtain
\[
\beta_U e^{z_{\Delta,2}}=(\beta_L+\beta_U)e^{z_{\Delta,1}+\tau}-\beta_L e^{\tilde{z}_1}+a(1-e^{-\sigma}) e^{\Delta+\sigma}.
\]
Since
\[
\beta_U+\beta_L-\beta_{L}e^{-\tau}=\beta_U+\overline{x}=\beta_Ue^{\tilde{z}_2-t_{\max}}
\]
and $\tilde{z}_2-t_{\max}+\tau=\tilde{z}_2-\tilde{z}_1$, we arrive at
\begin{equation}\label{e:betaszeros}
\beta_L e^{\tilde{z}_1}+\beta_U e^{\tilde{z}_2}=(\beta_L+\beta_U)e^{\tau+\tilde{z}_1}.
\end{equation}
Thus
\[
\begin{split}
\beta_U e^{z_{\Delta,2}}&=\beta_Ue^{\tilde{z}_2}+(\beta_L+\beta_U)e^{\tau}(e^{z_{\Delta,1}}-e^{\tilde{z}_1})+a(e^{\sigma}-1) e^{\Delta} \\
&=\beta_Ue^{\tilde{z}_2}+\dfrac{a(\beta_L+\beta_U)e^{\tau}}{\beta_L+a}(e^{\Delta}-e^{\tilde{z}_1})+a(e^{\sigma}-1) e^{\Delta}\quad\text{(with \eqref{d:rez1p})},
\end{split}
\]
which implies the formula for $T(\Delta)=z_{\Delta,2}+\tau$, since  $\tilde{z}_2+\tau=\tilde{T}$.
\end{proof}

In the case $\mathbf{RPRP}$ we always have $x^{(\Delta)}(\Delta+\sigma)>0$, but in the following proof we need to distinguish between two cases
\begin{description}
\item[Case $\mathbf{RPRP1}$] $0< x^{(\Delta)}(\Delta+\sigma)\le \beta_L$, and
\item[Case $\mathbf{RPRP2}$] $x^{(\Delta)}(\Delta+\sigma)> \beta_L$.
\end{description}

\begin{proof}[Proof of Proposition \ref{c:prop3}]
Using \eqref{e:after} for $\Delta+\sigma\le t\le z_{\Delta,1}+\tau=\tilde{z}_1+\tau=t_{\max}$ and \eqref{e:risingxDeltasigma} we obtain
\begin{align*}
 x^{(\Delta)}(z_{\Delta,1}+\tau) & = \beta_L+\bigl(x^{(\Delta)}(\Delta+\sigma)-\beta_L\bigr)e^{-\left(\tilde{z}_1+\tau-(\Delta+\sigma)\right)}\\
& =  \beta_L-\beta_Le^{-\tau} +a(1-e^{-\sigma})e^{-\left(\tilde{z}_1+\tau-(\Delta+\sigma)\right)}\\
& =   \overline{x} +a(e^{\sigma}-1)e^{-(\tilde{z}_1+\tau)+\Delta}>\overline{x},
\end{align*}
and $ x^{(\Delta)}(z_{\Delta,1}+\tau)$ is increasing as a function of $\Delta\in I_{RPRP}$. Also $x^{(\Delta)}(\Delta+\sigma)$ is increasing as a function of $\Delta\in I_{RPRP}$. In case $\mathbf{RPRP1}$ the function $x^{(\Delta)}$ is increasing on $[\tilde{z}_1,\tilde{z}_1+\tau]$ with $x^{(\Delta)}(t)\le\beta_L$ on this interval, hence $\max_{\tilde{z}_1\le t\le \tilde{z}_1+\tau}x^{(\Delta)}(t)=x^{(\Delta)}(\tilde{z}_1+\tau)\le\beta_L$. In case $\mathbf{RPRP2}$ the function $x^{(\Delta)}$ is increasing on $[\tilde{z}_1,\Delta+\sigma]$ and decreasing on $[\Delta+\sigma,\tilde{z}_1+\tau]$, with $x^{(\Delta)}(t)>\beta_L$ on this interval, hence $\max_{\tilde{z}_1\le t\le \tilde{z}_1+\tau}x^{(\Delta)}(t)=x^{(\Delta)}(\Delta+\sigma)>\beta_L$.
In both subcases,
$$
\max_{\tilde{z}_1\le t\le \tilde{z}_1+\tau}x^{(\Delta)}(t)=\max\{x^{(\Delta)}(\Delta+\sigma),x^{(\Delta)}(\tilde{z}_1+\tau)\}\ge x^{(\Delta)}(\tilde{z}_1+\tau)>\overline{x},
$$
and $\max_{\tilde{z}_1\le t\le \tilde{z}_1+\tau}x^{(\Delta)}(t)$ is increasing as a function of $\Delta\in I_{RPRP}$.

As $x^{(\Delta)}(t)>0$ on $(\tilde{z}_1,\tilde{z}_1+\tau]$ we have
$$
x^{(\Delta)}(t)=-\beta_U+\bigl(x^{(\Delta)}(\tilde{z}_1+\tau)+\beta_U\bigr)e^{-\left(t-(\tilde{z}_1+\tau)\right)}
$$
for $t\ge\tilde{z}_1+\tau$ as long as $x^{(\Delta)}(t-\tau)>0$. It follows that there is a smallest zero $z_{\Delta,2}$ of $x^{(\Delta)}$ in $(\tilde{z}_1+\tau,\infty)$, and
$$
x^{(\Delta)}(z_{\Delta,2}+t)=\tilde{x}(\tilde{z}_0+t)\quad\text{for all}\quad t\ge0.
$$
This yields
$$
T(\Delta)=\bigl(z_{\Delta,2}+(\tilde{z}_1-\tilde{z}_0)\bigr)-\tilde{z}_1=z_{\Delta,2}+\tau.
$$
Moreover,
\begin{align*}
\overline{x}_{\Delta} & = \max_{\tilde{z}_1\le t\le z_{\Delta,2}+(\tilde{z}_1-\tilde{z}_0)}x^{(\Delta)}(t)\\
& = \max_{\tilde{z}_1\le t\le\tilde{z}_1+\tau}x^{(\Delta)}(t)\\
& =  \max\{x^{(\Delta)}(\Delta+\sigma),x^{(\Delta)}(\tilde{z}_1+\tau)\}\ge\overline{x}
\end{align*}
is increasing as a function of $\Delta\in I_{RPRP}$, and $\underline{x}_{\Delta}=\underline{x}$.

Recall $t_{\max}=\tilde{z}_1+\tau$. From
$$
0=x^{(\Delta)}(z_{\Delta,2})=-\beta_U+\bigl(x^{(\Delta)}(t_{\max})+\beta_U\bigr)e^{-(z_{\Delta,2}-t_{\max})}
$$
we get
\[
\begin{split}
\beta_U e^{z_{\Delta,2}}&=\bigl(x^{(\Delta)}(t_{\max})+\beta_U\bigr)e^{t_{\max}}\\
&=(\overline x+\beta_U)e^{t_{\max}}+a(e^{\sigma}-1)e^{\Delta}.
\end{split}
\]
Since $(\overline{x}+\beta_U)e^{t_{\max}}=\beta_{U}e^{\tilde{z}_2}$, we conclude that
\[
\beta_U e^{z_{\Delta,2}}=\beta_U e^{\tilde{z}_2} +a(e^{\sigma}-1)e^{\Delta}.
\]
For the cycle length we obtain $T(\Delta)=z_{\Delta,2}+\tau>\tilde{z}_2+\tau=\tilde{T}$ and the formula for $T(\Delta)$ follows.
\end{proof}

\begin{proof}[Proof of Remark \ref{rem:rem1}]
\begin{enumerate}
\item Let $\delta_1>0$. We first show that the expressions defining $T(\Delta)$ in Proposition \ref{c:prop1} and in Proposition \ref{c:prop2} yield the same value for $\Delta=\delta_1$.

Consider the argument of $\ln$ in \eqref{eqn:prop5.2}.  We have
\begin{multline*}
\dfrac{a(e^{\sigma}-1)}{\beta_U} e^{\Delta-\tilde{z}_2}+\dfrac{a(\beta_L+\beta_U)e^{\tau+\tilde{z}_1-\tilde{z}_2}}{\beta_U(\beta_L+a)}
(e^{\Delta-\tilde{z}_1}-1)\\
=\dfrac{ae^{-\tilde{z}_2}}{\beta_U(\beta_L+a)} \Bigl(\bigl((\beta_L+a)(e^{\sigma}-1)+(\beta_L+\beta_U)e^{\tau}\bigr)e^{\Delta}
-(\beta_L+\beta_U)e^{\tau+\tilde{z}_1}\Bigr).
\end{multline*}
From \eqref{d:delta1} for $\delta_1$ it follows that
\[
\beta_Le^\sigma+a(e^\sigma-1)=\beta_Le^{\tilde{z}_1-\delta_1},
\]
which gives
\[
(\beta_L+a)(e^{\sigma}-1)=\beta_Le^{\tilde{z}_1-\delta_1}-\beta_L.
\]
Since
\[
-\beta_L+(\beta_L+\beta_U)e^{\tau}=e^{\tau}(\beta_U+\overline{x})=\beta_Ue^{\tilde{z}_2-\tilde{z}_1},
\]
we obtain
\begin{multline*}
\dfrac{a(e^{\sigma}-1)}{\beta_U} e^{\Delta-\tilde{z}_2}+\dfrac{a(\beta_L+\beta_U)e^{\tau+\tilde{z}_1-\tilde{z}_2}}{\beta_U(\beta_L+a)}
(e^{\Delta-\tilde{z}_1}-1)\\
=\dfrac{ae^{-\tilde{z}_2}}{\beta_U(\beta_L+a)} \left((\beta_Le^{\tilde{z}_1-\delta_1}+\beta_Ue^{\tilde{z}_2-\tilde{z}_1})e^{\Delta}
-(\beta_Le^{\tilde{z}_1}+\beta_Ue^{\tilde{z}_2})\right),
\end{multline*}
which for $\Delta=\delta_1$ becomes
\[
\dfrac{ae^{-\tilde{z}_2}}{\beta_U(\beta_L+a)} \left((\beta_Le^{\tilde{z}_1-\delta_1}+\beta_Ue^{\tilde{z}_2-\tilde{z}_1})e^{\Delta}
-(\beta_Le^{\tilde{z}_1}+\beta_Ue^{\tilde{z}_2})\right)=\dfrac{a(e^{\delta_1-\tilde{z}_1}-1)}{\beta_L+a}.
\]
We have
\[
\begin{split}
e^{\delta_1-\tilde{z}_1}-1&=\dfrac{\beta_L}{\beta_Le^{\sigma}+a(e^{\sigma}-1)}-1
=-\dfrac{(\beta_L+a)(e^{\sigma}-1)}{\beta_Le^{\sigma}+a(e^{\sigma}-1)}
\end{split},
\]
which leads to
\[
\begin{split}
\dfrac{a(e^{\delta_1-\tilde{z}_1}-1)}{\beta_L+a}&=-\dfrac{a(e^{\sigma}-1)}{\beta_Le^{\sigma}+a(e^{\sigma}-1)}
=-\dfrac{a(e^{\sigma}-1)}{\beta_L}e^{\delta_1-\tilde{z}_1}
\end{split}
\]
and shows that the formulae for $T(\Delta)$ from Propositions 5.1 and 5.2 yield the same value for $\Delta=\delta_1>0$.
From \eqref{eqn:prop5.1} of Proposition \ref{c:prop1}, this value is strictly less than $\tilde{T}$.

\item By continuity, we infer $T(\Delta)<\tilde{T}$ for $\Delta$ close to $\delta_1>0$.
\end{enumerate}
\end{proof}

\begin{proof}[Proof of Proposition \ref{c:prop4}]
Since $x^{(\Delta)}(t)>0$ for $\tilde{z}_1<t<\Delta+\sigma$ and $\tilde{z}_1+\tau<\Delta+\sigma$ we obtain that on $[\Delta+\sigma,\infty)$,
$$
x^{(\Delta)}(t)=-\beta_U+\bigl(x^{(\Delta)}(\Delta+\sigma)+\beta_U\bigr)e^{-\left(t-(\Delta+\sigma)\right)}\quad\text{as long as}\quad 0<x^{(\Delta)}(t-\tau).
$$
As $-\beta_U<0$ there is a smallest zero $z_{\Delta,2}$ of $x^{(\Delta)}$ in $[\Delta+\sigma,\infty)$, and
$$
x^{(\Delta)}(z_{\Delta,2}+t)=\tilde{x}(\tilde{z}_2+t)\quad\text{for all}\quad t\ge0.
$$
It follows that
$$
T(\Delta)=z_{\Delta,2}+(\tilde{z}_3-\tilde{z}_2)-\tilde{z}_1=z_{\Delta,2}+\tilde{T}-\tilde{z}_2=z_{\Delta,2}-\tilde{z}_0=z_{\Delta,2}+\tau.
$$
Moreover, $\overline{x}_{\Delta}=x^{(\Delta)}(t_{\max})$ if $\Delta\le \hat{\delta}_2$ (in which case $x^{(\Delta)}$ is decreasing on $[t_{\max},\Delta+\sigma]$), while for $\Delta>\hat{\delta}_2$ the function $x^{(\Delta)}$ is increasing on $[t_{\max},\Delta+\sigma]$ and $\overline{x}_{\Delta}=x^{(\Delta)}(\Delta+\sigma)>x^{(\Delta)}(t_{\max})\ge\overline{x}$. Hence $\overline{x}_{\Delta}\ge\overline{x}$. Obviously, $\underline{x}_{\Delta}=\underline{x}$. Also, $\overline{x}_{\Delta}$ is strictly decreasing as a function of $\Delta\in I_{RPFP}$, see \eqref{e:xprf}.

From
$$
0=x^{(\Delta)}(z_{\Delta,2})=-\beta_U+\bigl(x^{(\Delta)}(\Delta+\sigma)+\beta_U\bigr)
e^{-(z_{\Delta,2}-(\Delta+\sigma))}
$$
it follows that
\begin{align*}
\beta_U e^{z_{\Delta,2}} & = \bigl(x^{(\Delta)}(\Delta+\sigma)+\beta_U\bigr)e^{
\Delta+\sigma}\\
 & = \bigl(\beta_Ue^{\tilde{z}_2-(
\Delta+\sigma)} +a(1-e^{-\sigma})\bigr)e^{
\Delta+\sigma}\quad\text{(see \eqref{e:xprf})}\\
 & =  \beta_Ue^{\tilde{z}_2}+a(e^{\sigma}-1)e^{\Delta}.
\end{align*}
Hence $z_{\Delta,2}>\tilde{z}_2$, and thereby
$$
T(\Delta)=z_{\Delta,2}+\tau>\tilde{z}_2+\tau=\tilde{T}
$$
Furthermore,
\begin{align*}
T(\Delta) & =  z_{\Delta,2}+\tau=\tilde{T}+z_{\Delta,2}-\tilde{z}_2\\
& =  \tilde{T}+\ln\left(1+\dfrac{a(e^{\sigma}-1)e^{\Delta-\tilde{z}_2}}{\beta_U}\right),
\end{align*} and the map
$$
I_{RPFP}\ni\Delta\mapsto T(\Delta)\in\mathbb{R}
$$
is strictly increasing.
\end{proof}

\begin{proof}[Proof of Proposition \ref{c:prop5}]

From $x^{(\Delta)}(\Delta+\sigma)<0<x^{(\Delta)}(t_{\max})$ we know that $x^{(\Delta)}$ is strictly decreasing on $[t_{\max},\Delta+\sigma]$. It follows that there is a single zero in this interval, which is given by
$$
0=x^{(\Delta)}(z_{\Delta,2})=-\beta_U+a+\bigl(x^{(\Delta)}(t_{\max})+\beta_U-a\bigr)e^{-(z_{\Delta,2}-t_{\max})},
$$
or equivalently,
\[
(\beta_U-a)e^{z_{\Delta,2}}  =  \bigl(x^{(\Delta)}(t_{\max})+\beta_U-a\bigr)e^{t_{\max}}.
\]
Also,
\begin{align*}
(\beta_U-a)e^{z_{\Delta,2}} & =  \bigl(\overline x+a(1-e^{\Delta-t_{\max}})+\beta_U-a\bigr)e^{t_{\max}}\\
& =  (\overline x+\beta_U)e^{t_{\max}}-ae^{\Delta},
\end{align*}
and we arrive at
\begin{equation}\label{e:z2P}
(\beta_U-a)e^{z_{\Delta,2}}=  \beta_U e^{\tilde{z}_2}-ae^{\Delta}.
\end{equation}
Since $\beta_U>a$ we infer that  the map $I_{RPFN}\ni\Delta\mapsto e^{z_{\Delta,2}-\Delta}\in\mathbb{R}$ is strictly decreasing.
%Since $\beta_U>a$ in the present case we infer that  the map $I_{RPFN}\ni\Delta\mapsto e^{z_{\Delta,2}-\Delta}\in\mathbb{R}$ is strictly decreasing.

We have $z_{\Delta,2}<\Delta+\sigma\le z_{\Delta,2}+\tau$, and on $[\Delta+\sigma,z_{\Delta,2}+\tau]$,
$$
x^{(\Delta)}(t)=-\beta_U+\bigl(x^{(\Delta)}(\Delta+\sigma)+\beta_U\bigr)e^{-(t-(\Delta+\tau))}
$$
is strictly decreasing and negative because we have $x^{(\Delta)}(\Delta+\sigma)+\beta_U>0$ from \eqref{e:xprf}. It follows that $x^{(\Delta)}$ is strictly increasing  on $[z_{\Delta,2}+\tau,\infty)$ as long as $x^{(\Delta)}(t-\tau)<0$. There is a smallest zero $z_{\Delta,3}$ of $x^{(\Delta)}$ in this interval, and
\begin{align*}
x^{(\Delta)}(z_{\Delta,3}+t) & =  \tilde{x}(\tilde{z}_1+t)\quad\text{for all}\quad t\ge0,\\
T(\Delta) & =  z_{\Delta,3}-\tilde{z}_1,\\
\overline{x}_{\Delta} & =  x^{(\Delta)}(t_{\max})\ge\overline{x},\\
\underline{x}_{\Delta} & =  x^{(\Delta)}(z_{\Delta,2}+\tau).
\end{align*}

We compute
\[
\underline{x}_{\Delta}=x^{(\Delta)}(z_{\Delta,2}+\tau)  =  -\beta_U+\bigl(x^{(\Delta)}(\Delta+\sigma)+\beta_U\bigr)
e^{-\left(z_{\Delta,2}+\tau-(\Delta+\sigma)\right)}
\]
and use
\[
x^{(\Delta)}(\Delta+\sigma)+\beta_U=\beta_U e^{\tilde{z}_2-(\Delta+\sigma)}+a(1-e^{-\sigma})>0
\]
from \eqref{e:xprf}. This gives
$$
\underline{x}_{\Delta}=-\beta_U+\bigl(\beta_U e^{\tilde{z}_2} -a e^{\Delta}\bigr)e^{-(z_{\Delta,2}+\tau)} +ae^{-\left(z_{\Delta,2}+\tau-(\Delta+\sigma)\right)}.
$$
With \eqref{e:z2P} we obtain
\begin{align*}
\underline{x}_{\Delta} & =  -\beta_U+(\beta_U-a)e^{-\tau} +ae^{-\left(z_{\Delta,2}+\tau-(\Delta+\sigma)\right)}\\
& =  -\beta_U+\beta_U e^{-\tau}+ae^{-\tau}(e^{-z_{\Delta,2}+\Delta+\sigma}-1).
\end{align*}
Since $\Delta+\sigma>z_{\Delta,2}$ we conclude that
\[
\underline{x}_{\Delta}>-\beta_U+\beta_U e^{-\tau}=\underline{x}.
\]

We turn to the cycle length $T(\Delta)=z_{\Delta,3}-\tilde{z}_1$.
The equation for $z_{\Delta,3}$, namely,
\begin{align*}
0 & = x^{(\Delta)}(z_{\Delta,3})=\beta_L+\bigl(x^{(\Delta)}(z_{\Delta,2}+\tau)-\beta_L\bigr)e^{-\left(z_{\Delta,3}-(z_{\Delta,2}+\tau)\right)}\\
& =  \beta_L+(\underline{x}_{\Delta}-\beta_L)e^{-\left(z_{\Delta,3}-(z_{\Delta,2}+\tau)\right)}
\end{align*}
yields
\begin{align*}
\beta_Le^{z_{\Delta,3}} & =  (\beta_L-\underline{x}_{\Delta})e^{z_{\Delta,2}+\tau}\\
& =  (\beta_L+\beta_U)e^{z_{\Delta,2}+\tau}-(\beta_U e^{\tilde{z}_2} -a e^{\Delta})- ae^{\Delta+\sigma}\\
& \quad \text{(with the formula for}\quad\underline{x}_{\Delta}\quad\text{and \eqref{e:z2P})}\\
& =  (\beta_L+\beta_U)e^{z_{\Delta,2}+\tau}-\beta_Ue^{\tilde{z}_2}-a(e^{\sigma}-1)e^{\Delta}.
\end{align*}
Since $\tilde{z}_3=\tilde{z}_2+\tau+\tilde{z}_1$, we have
\[
\begin{split}
\beta_Le^{\tilde{z}_3}+\beta_Ue^{\tilde{z}_2}&=(\beta_L e^{\tilde{z}_1}+\beta_U e^{-\tau})e^{\tilde{z}_2+\tau}\\
&=(\beta_L+\beta_U)e^{\tilde{z}_2+\tau},
\end{split}
\]
which gives
\[
\begin{split}
\beta_Le^{z_{\Delta,3}}&= \beta_Le^{\tilde{z}_3}+(\beta_L+\beta_U)e^{\tau}(e^{z_{\Delta,2}}-e^{\tilde{z}_2})-a(e^{\sigma}-1)e^{\Delta}.
\end{split}
\]
Now use \eqref{e:z2P}  again to obtain
\[
\begin{split}
\beta_Le^{z_{\Delta,3}}&= \beta_Le^{\tilde{z}_3}-\dfrac{a(\beta_L+\beta_U)e^{\tau}}{\beta_U-a}
(e^{\Delta}-e^{\tilde{z}_2})-a(e^{\sigma}-1)e^{\Delta},
\end{split}
\]
which implies the formula for $T(\Delta)$.
We have
\[
\dfrac{(\beta_L+\beta_U)}{\beta_U-a}e^{\tau}+e^{\sigma}-1
=\dfrac{\beta_Le^{\tilde{z}_1+\tau}+(\beta_U-a)e^{\sigma}+a}{\beta_U-a}.
\]
Thus,  the map $I_{RPFN} \ni\Delta\mapsto T(\Delta)\in\mathbb{R}$ is strictly decreasing, since $\beta_U>a$.
\end{proof}

\begin{proof}[Proof of Remark \ref{rem:rem2}]
\begin{enumerate}

\item %Let $\delta_2<\infty$.
We first show that the expressions defining $T(\Delta)$ in Proposition \ref{c:prop3} and in Proposition \ref{c:prop5} (Equations \ref{e:TPbetaLB} and \ref{eqn:prop5.5} respectively) yield the same value for $\Delta=\delta_2$.

Consider the argument of $\ln$ in \eqref{eqn:prop5.5} of Proposition \ref{c:prop5}.
We have
\begin{multline*}
\dfrac{a(e^{\sigma}-1)}{\beta_L}e^{\Delta-\tilde{z}_1-\tilde{T}}+\dfrac{a(\beta_L+\beta_U)e^{-\tilde{z}_1}}
{\beta_L(\beta_U-a)}e^{\Delta-\tilde{z}_2}\\=
\dfrac{ae^{\Delta-\tilde{z}_1-\tilde{T}}}{\beta_L(\beta_U-a)}\Bigl((\beta_U-a)(e^{\sigma}-1)+(\beta_L+\beta_U)e^{\tau}\Bigr)
\end{multline*}
Note that  from the definition of $\delta_2$ it follows that
\[
\beta_Ue^{\sigma}-a(e^{\sigma}-1)=\beta_Ue^{\tilde{z}_2-\delta_2},
\]
which gives
\[
\begin{split}
(\beta_U-a)(e^{\sigma}-1)+(\beta_L+\beta_U)e^{\tau}&=\beta_Ue^{\sigma}-a(e^{\sigma}-1)+(\beta_L-\underline{x})e^{\tau}\\
&=\beta_Ue^{\tilde{z}_2-\delta_2}+\beta_L e^{t_{\max}}
\end{split}
\]
and leads to
\begin{multline*}
\dfrac{a(e^{\sigma}-1)}{\beta_L}e^{\Delta-\tilde{z}_1-\tilde{T}}+\dfrac{a(\beta_L+\beta_U)e^{-\tilde{z}_1}}
{\beta_L(\beta_U-a)}(e^{\Delta-\tilde{z}_2}-1)\\ =
\dfrac{ae^{-\tilde{z}_1}}{\beta_L(\beta_U-a)}\Bigl((\beta_Ue^{-\tau-\delta_2}+\beta_L e^{\tilde{z}_1-\tilde{z}_2})e^{\Delta}-(\beta_L+\beta_U)\Bigr).
\end{multline*}
Observe that for $\Delta=\delta_2$ we have
\begin{equation*}
(\beta_Ue^{-\tau-\delta_2}+\beta_L e^{\tilde{z}_1-\tilde{z}_2})e^{\Delta}= \beta_U e^{-\tau}+\beta_Le^{\tilde{z}_1+\delta_2-\tilde{z}_2}.
\end{equation*}
Using this we obtain that for $\Delta=\delta_2$,
\begin{multline*}
\dfrac{a(e^{\sigma}-1)}{\beta_L}e^{\Delta-\tilde{z}_1-\tilde{T}}+\dfrac{a(\beta_L+\beta_U)e^{-\tilde{z}_1}}
{\beta_L(\beta_U-a)}(e^{\Delta-\tilde{z}_2}-1)\\ =
\dfrac{a}{(\beta_U-a)}(e^{\delta_2-\tilde{z}_2}-1).
\end{multline*}
We have
\begin{align*}
\dfrac{a}{(\beta_U-a)}(e^{\delta_2-\tilde{z}_2}-1)&
=\dfrac{a}{(\beta_U-a)}\left(\dfrac{\beta_U}{\beta_Ue^{\sigma}-a(e^{\sigma}-1)}-1\right)\\
&=-\dfrac{a(e^{\sigma}-1)}{\beta_Ue^{\sigma}-a(e^{\sigma}-1)}\\
&=-\dfrac{a(e^{\sigma}-1)}{\beta_U}e^{\delta_2-\tilde{z}_2}.
\end{align*}
Thus, the formulae \eqref{e:TPbetaLB} and \eqref{eqn:prop5.5} from Propositions 5.3 and 5.5 are the same for $\Delta=\delta_2$.
Since
$$
-\dfrac{a(e^{\sigma}-1)}{\beta_U}e^{\delta_2-\tilde{z}_2}<0
$$
we also deduce that the value given by both equations for $\Delta=\delta_2$ is strictly larger than $\tilde{T}$.

\item By continuity, we infer $T(\Delta)>\tilde{T}$ for $\Delta$ close to $\delta_2$.
\end{enumerate}
\end{proof}

\begin{proof}[Proof of Proposition \ref{c:prop6}]
From $0\le x^{(\Delta)}(t)$ on $[t_{\max},\Delta+\sigma]$ we infer that $x^{(\Delta)}$ decreases on $[\Delta+\sigma,\infty)$ as long as $x^{(\Delta)}(t-\tau)>0$. This yields the existence of a smallest zero $z_{\Delta,2}$ in $[\Delta+\sigma,\infty)$, and
$x^{(\Delta)}(t+z_{\Delta,2})=\tilde{x}(t+\tilde{z}_2)$ for all $t\ge0$.

Hence
$$
T(\Delta)=z_{\Delta,2}+\tilde{z}_3-\tilde{z}_2-\tilde{z}_1=z_{\Delta,2}+\tilde{T}-\tilde{z}_2=z_{\Delta,2}+\tau.
$$
We have
$$
-\beta_U+\bigl(x^{(\Delta)}(\Delta+\sigma)+\beta_U\bigr)e^{-(z_{\Delta,2}-(\Delta+\sigma))}=0
$$
and conclude that
\[
\begin{split}
\beta_U e^{z_{\Delta,2}}&=\bigl(x^{(\Delta)}(\Delta+\sigma)+\beta_U\bigr)e^{\Delta+\sigma}\\
&=\bigl(\beta_Ue^{\tilde{z}_2-\sigma -\Delta}+a(1-e^{- \sigma })\bigr)e^{\Delta+\sigma}\\
&=\beta_Ue^{\tilde{z}_2}+a(e^{\sigma}-1)e^{\Delta},
\end{split}
\]
which implies \eqref{e:TPbetaLB}, the desired formula for $T(\Delta)$.
\end{proof}

\begin{proof}[Proof of Proposition \ref{c:prop7}]
Since $[\Delta,\Delta+\sigma]\subset(t_{\max},\tilde{T})$ we have
$$
x^{(\Delta)}(t)=-\beta_U+a+\bigl(x^{(\Delta)}(\Delta)+\beta_U-a\bigr)e^{-(t-\Delta)}\quad\text{on}\quad[\Delta,\Delta+\sigma].
$$
Using $x^{(\Delta)}(\Delta)+\beta_U-a>0$ we see that $x^{(\Delta)}$ is strictly decreasing on $[\Delta,\Delta+\sigma]$, and has a unique zero $z_{\Delta,2}$ in $[\Delta,\Delta+\sigma)$, which is given implicitly by
$$
0=x^{(\Delta)}(z_{\Delta,2})=-\beta_U+a+\bigl(x^{(\Delta)}(\Delta)+\beta_U-a\bigr)
e^{-(z_{\Delta,2}-\Delta)}.
$$
Combining this with \eqref{e:x2+DbU} gives
\[
(\beta_U-a)e^{z_{\Delta,2}}=\beta_Ue^{\tilde{z}_2}-ae^{\Delta},
\]
which is the same as \eqref{e:z2P}.
As $z_{\Delta,2}\ge\Delta$ we infer
\[
\beta_U(e^{z_{\Delta,2}}-e^{\tilde{z}_2})=a(e^{z_{\Delta,2}}-e^{\Delta})\ge0.
\]
Hence $z_{\Delta,2}\ge\tilde{z}_2$. Since $x^{(\Delta)}(\Delta)\ge 0$ and $\beta_U>a$, the map $[\Delta,\Delta+\sigma] \ni t\mapsto x^{(\Delta)}(t)$ is strictly decreasing. The rest of the proof is the same as the proof of Proposition \ref{c:prop5} starting  after \eqref{e:z2P}.
\end{proof}

\begin{proof}[Proof of Proposition \ref{c:prop8}]
\begin{enumerate}
\item We have $x^{(\Delta)}(\tilde{z}_2)=0$, and the function $x^{(\Delta)}$ is strictly decreasing on $[\tilde{z}_2,\Delta]$, monotone on $[\Delta,\Delta+\sigma]$ with $x^{(\Delta)}(t)<0$ for $\Delta\le t<\Delta+\sigma$, and strictly decreasing on $[\Delta+\sigma,\tilde{T}]$. It follows that $x^{(\Delta)}$ is strictly increasing on $[\tilde{T},\infty)$ as long as $x^{(\Delta)}(t-\tau)<0$. This yields a first zero $z_{\Delta,3}$ of $x^{(\Delta)}$ in $[\tilde{T},\infty)$ and $x^{(\Delta)}(z_{\Delta,3}+t)=\tilde{x}(\tilde{z}_1+t)$ for all $t\ge0$, and thus
$$
T(\Delta)=z_{\Delta,3}+(\tilde{z}_2-\tilde{z}_1)-\tilde{z}_2=z_{\Delta,3}-\tilde{z}_1,
$$
$$
\underline{x}_{\Delta}=\min\{x^{(\Delta)}(\Delta),x^{(\Delta)}(\tilde{T})\},
$$
and
$$
\overline{x}_{\Delta}=\overline{x}.
$$
Observe that
\[
\begin{split}
x^{(\Delta)}(\tilde{T})&=-\beta_U+\beta_U e^{-\tau}+a ( e^{ \sigma }-1)e^{\Delta-\tilde{T}}\\
&=\underline{x}+a ( e^{ \sigma }-1)e^{\Delta-\tilde{T}}>\underline{x}.
\end{split}
\]
Using this and $x^{(\Delta)}(\Delta)=\tilde{x}(\Delta)>\underline{x}$ we have $\underline{x}_{\Delta}>\underline{x}$.
The equation
$$
0=x^{(\Delta)}(z_{\Delta,3})=\beta_L+\bigl(x^{(\Delta)}(\tilde{T})-\beta_L\bigr)e^{-(z_{\Delta,3}-\tilde{T})}
$$
is equivalent to
$$
\beta_L e^{z_{\Delta,3}}=\bigl(\beta_L-x^{(\Delta)}(\tilde{T})\bigr)e^{\tilde{T}},
$$
which gives
\[
\begin{split}
\beta_L e^{z_{\Delta,3}}&=(\beta_L+\beta_U)e^{\tilde{T}}-\bigl(\beta_U e^{-\tau}+a ( e^{ \sigma }-1)e^{\Delta-\tilde{T}}\bigr)e^{\tilde{T}}\\
&=(\beta_L+\beta_U-\beta_Ue^{-\tau}) e^{\tilde{T}}-a ( e^{ \sigma }-1)e^{\Delta}.
\end{split}
\]
We have $\beta_L+\beta_U-\beta_Ue^{-\tau}=\beta_L e^{\tilde{z}_1}$
and we conclude  that
\[
\beta_L e^{z_{\Delta,3}}=\beta_Le^{\tilde{z}_1+\tau+\tilde{z}_2}-a(e^{\sigma}-1)e^{\Delta},
\]
from which we obtain \eqref{e:prop5.9} for $T(\Delta)=z_{\Delta,3}-\tilde{z}_1$.

\item Observe that for every $\Delta\in I_{FNFN}$ we have
$$
x^{(\Delta)}(\tilde{T})-\tilde{x}(\Delta)=g(\Delta)
$$
with the strictly increasing function
$$
g:\mathbb{R}\ni\Delta\mapsto a(e^{\sigma}-1)e^{\Delta-\tilde{T}}+\beta_Ue^{-\tau}-\beta_Ue^{\tilde{z}_2-\Delta}\in\mathbb{R}
$$
which has a single zero at $\Delta=\overline{\delta}$ since $g(\Delta)\to-\infty$ as $\Delta\to-\infty$,
$g(\Delta)\to\infty$ as $\Delta\to\infty$, and $g(\Delta)=0$ if and only if
$$
\beta_U(e^{\tilde{T}-\Delta})^2-\beta_Ue^{\tilde{T}-\Delta}-a(e^{\sigma}-1)e^{\tau}=0,
$$
or equivalently,
\begin{align*}
e^{\tilde{T}-\Delta} & =  \dfrac{1}{2}+\sqrt{\dfrac{1}{4}+\dfrac{a(e^{\sigma}-1)e^{\tau}}{\beta_U}}\\
& =  \dfrac{\beta_U+\sqrt{\beta_U ^2 +4a\beta_U(e^{\sigma}-1)e^{\tau}}}{2\beta_U}.
\end{align*}

For $\Delta<\overline{\delta}$ in $I_{FNFN}$ we have $g(\Delta)<0$, hence
$$
\underline{x}_{\Delta}=\min\{\tilde{x}(\Delta),x^{(\Delta)}(\tilde{T})\}=x^{(\Delta)}(\tilde{T}),
$$
and the formula \eqref{e:xprf} for $x^{(\Delta)}(\tilde{T})$ in Part 1 above shows that in case $I_{FNFN}\cap(-\infty,\overline{\delta})\neq\emptyset$ the map
$$
I_{FNFN}\cap(-\infty,\overline{\delta})\ni\Delta\mapsto\underline{x}_{\Delta}\in\mathbb{R}
$$
is strictly increasing. For $\Delta>\overline{\delta}$ in $I_{FNFN}$ we get $0<g(\Delta)$, hence
$$
\underline{x}_{\Delta}=\min\{\tilde{x}(\Delta),x^{(\Delta)}(\tilde{T})\}=\tilde{x}(\Delta),
$$
and we see that in case $I_{FNFN}\cap(\overline{\delta},\infty)\neq\emptyset$ the map
$$
I_{FNFN}\cap(\overline{\delta},\infty)\ni\Delta\mapsto\underline{x}_{\Delta}\in\mathbb{R}
$$
is strictly decreasing. \qedhere
\end{enumerate}
\end{proof}

\begin{proof}[Proof of Remark~\ref{rem:rem3}]
The strictly increasing function $g$ from the previous proof satisfies
\begin{align*}
g(\tilde{z}_2) & =  a(e^{\sigma}-1)e^{-\tau}+\beta_U(e^{-\tau}-1)\\
& <  \beta_U\bigl((e^{\sigma}-1)e^{-\tau}+e^{-\tau}-1\bigr)\quad\text{(with}\quad a<\beta_U)\\
& =  \beta_U(e^{\sigma-\tau}-1)\\
& \le   0\quad\text{(since }\sigma\le\tau)
\end{align*}
and
$$
g(\tilde{T}-\sigma)=a(e^{\sigma}-1)e^{-\sigma}+\beta_Ue^{-\tau}-\beta_Ue^{\tilde{z}_2-\tilde{T}+\sigma}>0
$$
if and only if $\beta_U e^{\sigma - \tau } < a$.
By the intermediate value theorem the only zero $\overline{\delta}$ of $g$ belongs to the interval $(\tilde{z}_2,\tilde{T}-\sigma)$ in case $ \beta_U e^{\sigma - \tau } < a$. Otherwise $\overline{\delta}\ge \tilde{T}-\sigma$.
\end{proof}

\begin{proof}[Proof of Proposition \ref{c:prop9}]
From $0>x^{(\Delta)}(\Delta+\sigma)>x^{(\Delta)}(\tilde{T})$ we obtain $x^{(\Delta)}(t)<0$ on $(\tilde{z}_2,\Delta+\sigma]$. It follows that
on $[\Delta+\sigma,\infty)$ the function $x^{(\Delta)}$ is strictly increasing as long as $x^{(\Delta)}(t-\tau)<0$, and there is a smallest zero $z_{\Delta,3}$ of $x^{(\Delta)}$ in $[\Delta+\sigma,\infty)$. Moreover, $x^{(\Delta)}(z_{\Delta,3}+t)=\tilde{x}(\tilde{z}_1+t)$ for all $t\ge0$, hence
$$
T(\Delta)=z_{\Delta,3}+(\tilde{z}_2-\tilde{z}_1)-\tilde{z}_2=z_{\Delta,3}-\tilde{z}_1,
$$
and
$$
\underline{x}_{\Delta}=x^{(\Delta)}(\tilde{T})>\underline{x},
$$
and $\overline{x}_{\Delta}=\overline{x}$. As the map $[\tilde{T}-\sigma,\tilde{T})\ni\Delta\mapsto x^{(\Delta)}(\tilde{T})\in\mathbb{R}$
 is strictly decreasing we infer that also the map $I_{FNRN}\ni\Delta\mapsto\underline{x}_{\Delta}\in\mathbb{R}$ is strictly decreasing.

From
$$
0=x^{(\Delta)}(z_{\Delta,3})=\beta_L+\bigl(x^{(\Delta)}(\Delta+\sigma)-\beta_L\bigr)
e^{-(z_{\Delta,3}-(\Delta+\sigma))}
$$
we obtain
\begin{align*}
\beta_L e^{z_{\Delta,3}} & =  \bigl(\beta_L- x^{(\Delta)}(\Delta+\sigma)\bigr)e^{
\Delta+\sigma}\\
 & =  \bigl(\beta_L-\beta_L+\beta_Le^{\tilde{z}_1-(
\Delta+\sigma-\tilde{T})}-a(1-e^{-\sigma})\bigr)e^{
\Delta+\sigma}\\
 & = \beta_Le^{\tilde{z}_1+\tilde{T}}- a(e^{\sigma}-1)e^{
\Delta},
\end{align*}
and the formula for $T(\Delta)=z_{\Delta,3}-\tilde{z}_1$ follows.
\end{proof}

\begin{proof}[Proof of Proposition \ref{c:prop10}]
We have $x^{(\Delta)}(t)<0$ on $(\tilde{z}_2,\tilde{T}]$, and there is a first zero of $x^{(\Delta)}$ in $(\tilde{T},\Delta+\sigma]$, which is given by
$$
0=x^{(\Delta)}(z_{\Delta,3})=\beta_L+a+\bigl(x^{(\Delta)}(\tilde{T})-(\beta_L+a)\bigr)e^{-(z_{\Delta,3}-\tilde{T})},
$$
or equivalently,
\begin{align*}
(\beta_L+a)e^{z_{\Delta,3}} & = \bigl(\beta_L+a-x^{(\Delta)}(\tilde{T})\bigr)e^{\tilde{T}}\\
& =  \bigl(\beta_L+a-(\underline{x}+a-ae^{\Delta-\tilde{T}
})\bigr)e^{\tilde{T}}\\
& =  (\beta_L-\underline{x})e^{\tilde{T}}+ae^{\Delta}\\
&=\beta_L e^{\tilde{z}_1+\tilde{T}}+ae^{\Delta}.
\end{align*}
Incidentally, this shows that the map $I_{FNRP}\ni\Delta\mapsto z_{\Delta,3}-\Delta\in\mathbb{R}$ is strictly decreasing.

On $[z_{\Delta,3},\Delta+\sigma]$ the function $x^{(\Delta)}$ is strictly increasing. Since $x^{(\Delta)}(t-\tau)<0$ on $(\tilde{z}_2,z_{\Delta,3})$, we infer that $x^{(\Delta)}$ is strictly increasing on $[z_{\Delta,3},z_{\Delta,3}+\tau]$. Hence $0<x^{(\Delta)}(t)$ on $(z_{\Delta,3},z_{\Delta,3}+\tau]$. This implies that $x^{(\Delta)}$ is strictly decreasing on $[z_{\Delta,3}+\tau,\infty)$ as long as $x^{(\Delta)}(t-\tau)>0$, and that there is a first zero $z_{\Delta,4}$ in $(z_{\Delta,3}+\tau,\infty)$. We obtain
$x^{(\Delta)}(z_{\Delta,4}+t)=\tilde{x}(\tilde{z}_2+t)$ for all $t\ge0$, which yields
$$
T(\Delta)=z_{\Delta,4}-\tilde{z}_2,
$$
$\underline{x}_{\Delta}=\min\{\tilde{x}(\Delta),x^{(\Delta)}(\tilde{T})\}>\underline{x}$, and
$$
\overline{x}_{\Delta}=x^{(\Delta)}(z_{\Delta,3}+\tau).
$$
As in the proof of Proposition \ref{c:prop9} we infer that the map $I_{FNRP}\ni\Delta\mapsto\underline{x}_{\Delta}\in\mathbb{R}$ is strictly decreasing.

From the formula for $x^{(\Delta)}(\Delta+\sigma)$ we obtain
\begin{align*}
x^{(\Delta)}(z_{\Delta,3}+\tau) & =  \beta_L+\bigl(x^{(\Delta)}(\Delta+\sigma)-\beta_L\bigr)e^{-(z_{\Delta,3}+\tau-(\Delta+\sigma))}\\
&= \beta_L-\beta_Le^{\tilde{z}_1+\tilde{T}-(z_{\Delta,3}+\tau)}
+ a(e^{\sigma}-1)e^{\Delta}e^{-(z_{\Delta,3}+\tau)},
\end{align*}
which  can be rewritten as
\[
\overline{x}_{\Delta}=\beta_L-\bigl(\beta_Le^{\tilde{z}_1+\tilde{T}}+ae^{\Delta}\bigr)e^{-(z_{\Delta,3}+\tau)}
+ ae^{\sigma+\Delta}e^{-(z_{\Delta,3}+\tau)}.
\]
Hence,
\[
\begin{split}
\overline{x}_{\Delta}&=\beta_L-(\beta_L+a)e^{-\tau}+ae^{\sigma+\Delta}e^{-(z_{\Delta,3}+\tau)}\\
&=\overline{x} +ae^{-\tau}(e^{\sigma+\Delta-z_{\Delta,3}}-1)\\
&>\overline{x}.
\end{split}
\]
We also know that the map $I_{FNRP}\ni\Delta\mapsto \overline{x}_{\Delta}\in\mathbb{R}$ is strictly increasing.

From
$$
0=x^{(\Delta)}(z_{\Delta,4})=-\beta_U+\bigl(x^{(\Delta)}(z_{\Delta,3}+\tau)+\beta_U\bigr)e^{-(z_{\Delta,4}-(z_{\Delta,3}+\tau))}
$$
we have
\[
\beta_Ue^{z_{\Delta,4}}=
(\beta_L+\beta_U)e^{z_{\Delta,3}+\tau}-\beta_Le^{\tilde{z}_1+\tilde{T}}+a(e^\sigma-1)e^{\Delta}.
\]
Use \eqref{e:betaszeros} to obtain
\[
\beta_Le^{\tilde{z}_1+\tilde{T}}+\beta_U e^{\tilde{z}_2+\tilde{T}}=(\beta_L+\beta_U)e^{\tau+\tilde{z}_1+\tilde{T}}
\]
which, combined  with the previous equation, gives
\[
\beta_Ue^{z_{\Delta,4}}=\beta_U e^{\tilde{z}_2+\tilde{T}}+ (\beta_L+\beta_U)e^{\tau}\bigl(e^{z_{\Delta,3}}-e^{\tilde{z}_1+\tilde{T}}\bigr)+a(e^\sigma-1)e^{\Delta}.
\]
Now the formula for $z_{\Delta,3}$ leads to
\[
\beta_Ue^{z_{\Delta,4}}=\beta_U e^{\tilde{z}_2+\tilde{T}}+\dfrac{a(\beta_L+\beta_U)e^{\tau}}{\beta_L+a}
\bigl(e^{\Delta}-e^{\tilde{z}_1+\tilde{T}}\bigr)+a(e^\sigma-1)e^{\Delta}
\]
which yields the formula for $T(\Delta)=z_{\Delta,4}-\tilde{z}_2$.
\end{proof}

\bibliographystyle{spbasic} 
\bibliography{stemcleaned_v4}

\end{document}